\documentclass[reqno]{amsart}

\usepackage{mathrsfs, color}
\usepackage{amsfonts}
\usepackage{amssymb,amsmath}
\usepackage{amsthm}
\usepackage{float}

\usepackage{graphics}
\usepackage{graphicx}
\usepackage{epsfig}

\usepackage{algorithm} 
\usepackage{algorithmic} 
\usepackage{multirow} 
\usepackage{xcolor}
\DeclareMathOperator*{\argmin}{argmin} 

\newtheorem{theorem}{Theorem}[section]
\newtheorem{lemma}[theorem]{Lemma}

\theoremstyle{definition}

\theoremstyle{remark}
\newtheorem{remark}[theorem]{Remark}

\numberwithin{equation}{section}

\begin{document}

\title[Backward Fractional Diffusion]
 {Backward Problem for a Time-Space Fractional Diffusion Equation}

\author[J.X.Jia]{Junxiong Jia}
\address{Department of Mathematics and Statistics,
Xi'an Jiaotong University,
 Xi'an
710049, China;
Beijing Center for Mathematics and Information Interdisciplinary Sciences (BCMIIS)}
\email{jjx323@xjtu.edu.cn}
\thanks{}

\author[J. Peng]{Jigen Peng}
\address{Department of Mathematics and Statistics,
Xi'an Jiaotong University,
 Xi'an
710049, China;
Beijing Center for Mathematics and Information Interdisciplinary Sciences (BCMIIS)}
\email{jgpeng@xjtu.edu.cn}

\author[J. Gao]{Jinghuai Gao}
\address{School of Electronic and Information Engineering,
Xi'an Jiaotong University,
 Xi'an
710049, China;
Beijing Center for Mathematics and Information Interdisciplinary Sciences (BCMIIS)}
\email{jhgao@xjtu.edu.cn}

\author[Y. Li]{Yujiao Li}
\address{Department of Bioengineering,
Xi'an Jiaotong University,
 Xi'an
710049, China;}
\email{liyujiao.323@stu.xjtu.edu.cn}

\subjclass[2010]{}

\date{}


\begin{abstract}
In this paper, a backward problem for a time-space fractional diffusion process has been considered.
For this problem, we propose to construct the initial function by minimizing data residual error in fourier space domain
with variable total variation (TV) regularizing term which can protect the edges as TV regularizing term and reduce
staircasing effect. The well-posedness of this optimization problem is obtained under a very general setting.
Actually, we rewrite the time-space fractional diffusion equation as an abstract fractional differential
equation and deduce our results by using fractional semigroup theory, hence, our theoretical results can be applied to
other backward problems for the differential equations with more general fractional operator.
Then a modified Bregman iterative algorithm has been proposed to approximate
the minimizer. The new features of this algorithm is that the regularizing term altered in each step and
we need not to solve the complex Euler-Lagrange equation of variable TV regularizing term (just need to solve
a simple Euler-Lagrange equation). The convergence of this algorithm and the strategy of choosing parameters
are also obtained. Numerical implementations are provided to support our theoretical analysis to show
the flexibility of our minimization model.
\end{abstract}

\maketitle


\section{Introduction}
Diffusion phenomenon is ubiquitous in our physical world.
From the point of view of probability theory, applying centeral limit theorem to the random walk problem we can
derive diffusion equations. If we assume the distribution of particle jump is Gaussian, we will obtain
normal diffusion equations
\begin{align}\label{1_normalDiffusion}
\begin{split}
\left\{ \begin{array}{l}
\partial_{t} v(t,x) - \Delta v(t,x) = 0  \\
v(x,0) = u(x).
\end{array} \right.
\end{split}
\end{align}
If we assume the particle jump satisfy L\'{e}vy distribution, by continuous time random walk (CTRW) model, we will derive
time-space fractional diffusion equation (FDE) as follows
\begin{align}\label{1_timespacefrac}
\begin{split}
\left\{ \begin{array}{ll}
\partial_{t}^{\alpha} v(t,x) + (-\Delta)^{\beta}v(t,x) = 0 & (x,t) \in \mathbb{R}^{2} \times (0,\infty) \\
v(x,0) = u(x) & x \in \mathbb{R}^{2}
\end{array} \right.
\end{split}
\end{align}
with $\alpha \in (0,1]$, $\beta \in (1/2,1]$.
Here the time derivative is in Djrbashian-Caputo sense defined as follows
\begin{align}\label{1_definitionCaputo}
\begin{split}
\partial_{t}^{\alpha}v(t) = (g_{1-\alpha} * \partial_{t}v)(t) := \int_{0}^{t} g_{1-\alpha}(t-s)v(s) ds
\end{split}
\end{align}
where
\begin{align}\label{1_definitionofG}
g_{\gamma}(t) :=  \left\{ \begin{array}{ll}
\frac{1}{\Gamma(\gamma)}t^{\gamma - 1}, & t > 0, \\
0, & t \leq 0,
\end{array} \right.
\end{align}
with $\Gamma(\gamma)$ is the Gamma function.
Denote the fourier transform of function $v$ as $\mathcal{F}(v)$ or $\hat{v}$, the inverse fourier
transform as $\mathcal{F}^{-1}v$ or $\check{v}$.
Then the space fractional derivative $(-\Delta)^{\beta}$ can be defined by fourier transform
as $(-\Delta)^{\beta}v = \mathcal{F}^{-1}(|\xi|^{\beta}\hat{v})$.
Usually, we call this type fractional derivative operator as symmetric Riesz-Feller space fractional derivative operator.

Fractional time-space diffusion equation (\ref{1_timespacefrac}) attracts lots of researchers attention.
From the physical point of view, there are two long papers \cite{Metzler2000,Zaslavsky2002} provide a good summary.
From the stochastic point of view, there is a good book \cite{Meerschaert2011} which gives rigorous mathematical
deductions. From the functional analysis point of view, Peng, Li \cite{Peng2012} propose fractional semigroup, Li, Chen \cite{LiChen2010}
propose $\alpha$-resolvent operator to provide a general theory for the fractional abstract Cauchy problem which can be
applied to FDE (\ref{1_timespacefrac}) and some more general FDEs. B. Baeumer et al. \cite{bea0,bea1,bea2,bea3} propose the concept of stochastic
solutions for fractional evolution equations and study FDEs by using stochastic methods combined with operator semigroup theory.

In this paper, we focus on the backward problem for equation (\ref{1_timespacefrac}).
As mentioned in a recent tutorial \cite{RundellJin2015}, the mathematical theory of inverse problems for FDEs
is still in its infancy. However, there are already some pioneering work in this direction.
Cheng et al. \cite{Cheng2009} establish the uniqueness in an inverse problem for a one-dimensional fractional
diffusion equation.
Sakamoto and Yamamoto \cite{fracJMAA2011} establish the unique existence of weak solutions and the asymptotic
behavior as time $t$ goes to $\infty$; they also prove the stability in the backward problem in time and the
uniqueness in determining an initial value.
Liu and Yamamoto \cite{Liu2010} study a backward  problem for a time-fractional diffusion equation.
Zhang and Xu \cite{ZhangXu2011} investigate an inverse source problem for fractional diffusion equation, they obtain the
uniqueness of the inverse problem by analytic continuation and Laplace transform.
Zheng and Wei \cite{ZhengWei2010} study backward problem of space fractional diffusion equations, they
show that the problem is severely ill-posed and propose a regularization method.

Recently, Wang and Liu \cite{WangLiu2013} propose to use more general anomalous diffusion models to
describe the blurring effect, and the backward problems give the mathematical formulation for the de-blurring process
in image restoration. By using total variation regularization term, the discontinuity of the initial data
can be recovered. In Wang and Liu's paper, they use time-fractional diffusion models, here we intend to use a more
general time-space fractional diffusion model (\ref{1_timespacefrac}).
In this paper, we assume the initial data $u(x) \in \mathbb{R}^{2}$ with compact support in a bounded convex open
subset $\Omega$ of the plane with Lipschitz continuous boundary $\partial\Omega$,
which is a reasonable assumption in many applications.

Denote $g^{\delta}(x)$ in $\Omega$ to be the measurement data, our backward problem is to approximate $v(0,x) = u(x)$ from $g^{\delta}(x)$.
For some known error level $\delta > 0$, the noisy data of the exact gray level $g(x) := v(T,x)$ satisfying
\begin{align}\label{1_noiseLevel}
\|g^{\delta}(\cdot) - g(\cdot) \|_{L^{2}(\Omega)} \leq \delta.
\end{align}

In many applications of the backward diffusion problem, the initial distribution $u(x)$ is in general not smooth.
Because $v(T,x)$ generated from the Cauchy problem (\ref{1_timespacefrac}),
$v(t,x)$ need not have compact support as $u(x)$. Here, we only use the measurement data $v(T,x)$ in $\Omega$,
i.e., the values $v(T,x)$ outside of $\Omega$ have nothing to do with our reconstruction process, we can define
$v(T,x)$ for $x \in \Omega$ such that
\begin{align}\label{1_dataoutOmega}
\|g^{\delta}(\cdot) - g(\cdot)\|_{L^{2}(\mathbb{R}^{2}\backslash\Omega)} = 0.
\end{align}

Taking Fourier transform with respect to $x$ in (\ref{1_timespacefrac}), we obtain
\begin{align}\label{1_fourierfractional}
\partial_{t}^{\alpha}\hat{v}(t,\xi) = - |\xi|^{\beta}\hat{v}(t,\xi).
\end{align}
By using the Laplace transform with respect to $t$ in (\ref{1_fourierfractional}), we can establish the relation
between $u(x)$ and $v(T,x)$ in frequency domain as
\begin{align}\label{1_fourierSolution}
\hat{g}(\xi) = \hat{v}(T,\xi) = \hat{S}(\xi) \hat{u}(\xi), \quad \hat{S}(\xi) := E_{\alpha,1}(-|\xi|^{\beta}T^{\alpha}),
\end{align}
where $E_{\alpha,1}(\cdot)$ is the Mittag-Leffer function defined as
\begin{align}\label{1_defintionMittagLeffler}
E_{\alpha,\gamma}(z) := \sum_{k = 0}^{\infty} \frac{z^{k}}{\Gamma(\alpha k + \gamma)}, \quad z \in \mathbb{C}, \quad
\alpha > 0, \quad \gamma > 0,
\end{align}
which can be seen as a generalization of exponential function $e^{z}$.
Denote $Su = \mathcal{F}^{-1}(\hat{S}(\xi) \hat{u}(\xi))$, then we have
\begin{align}\label{1_timespacedomain}
g(x) = v(T,x) = (Su)(x).
\end{align}
Intuitively, the operator $S$ is a convolution operator with kernel $\mathcal{F}^{-1}(\hat{S}(\xi))$.
In part 3 of section 2, we will define the operator $S$ (formula (\ref{2_3defofS}))
by the solution operator of an abstract fractional evolution equation.

As mentioned by the previous works \cite{RundellJin2015,WangLiu2013}, recovering $u(x)$ from the noisy measurement
of exact $v(T,x)$ base on relation (\ref{1_fourierSolution}) in the frequency domain is ill posed due to the rapid
decay of the forward process. Usually, there are two conventional methods, namely, Tikhonov regularization and
truncated Fourier transform regularization, to overcome this difficulty in the frequency domain.
In 2013, Wang and Liu \cite{WangLiu2013} proposed to use total variation (TV) regularization for time fractional
diffusion model.

TV regularization method successfully recover the edges of the initial data and is robust for the noise.
However, TV regularization method suffers from staircasing effect, which is a noise induced introduction
of artificial steps or discontinuities into the reconstructed or denoised noise.
In order to reduce this effect, Blomgren et al. \cite{Blomgren2000} suggest letting the exponent
in the regularization term depend on the data.
Li et al. \cite{LiLiPi2010} studies the variable exponent TV regularization when exponent $1 < p(x) \leq 2$.
Harjulehto et al. \cite{Harjulehto2013} studies the variable TV regularization allowing $p(x) = 1$ for some $x$ by
using techniques development in \cite{Harjulehto2008}.
Bollt et al. \cite{var2009} studies the following variation model
\begin{align}\label{1_variationImage}
\min_{u}J(u) = \int_{\Omega}|\nabla u|^{\tilde{p}(x)} dx + \frac{\lambda}{2}\int_{\Omega}|u-g|^{2} dx,
\end{align}
where $\lambda > 0$ is a positive constant, $g$ is the noisy image, $\tilde{p}(\cdot)$ defined as
\begin{align}\label{1defp}
\tilde{p}(x) = p(u) = P_{M}(|\nabla(G_{\tilde{\delta}} * u)(x)|^2)
\end{align}
with $G_{\tilde{\delta}} : \mathbb{R}^{2} \rightarrow \mathbb{R}$ is a symmetric mollifier centered at $0$
and belongs to $C^{2} \cap W^{3,2}$.
$P_{M} : \mathbb{R}^{+} \rightarrow [1,2]$ is a non-increasing $C^{2}$ function with $P_{M}(M) = 1$ with $M > 0$ is a
positive real number. For example, $P_{M}$ can be taken as follows:
\begin{align}\label{1PM}
\begin{split}
P_{M}(s) = \left\{ \begin{array}{ll}
2 - \frac{10s^3}{M^3} + \frac{15s^4}{M^4} - \frac{6s^5}{M^5} & \text{if }s \leq M\\
1 & \text{if }s > M
\end{array} \right. .
\end{split}
\end{align}
They studies how parameter choices affect recover results and prove the existence and uniqueness of minimizers.
Recently, in \cite{wuchi}, the author studies image decomposition problems
by using variable total regularization combined with variable Besov space.

We attempt to use the variable total variation regularization term to penalize our fractional backward diffusion problem.
More specifically, we intend to use the following model
\begin{align}\label{1Question}
u^{\delta} = \argmin_{u \in \mathcal{K}} \int_{\Omega}|\nabla u|^{\tilde{p}(x)} dx \,\, \text{such that} \,\,
\|S u - g^{\delta}\|^{2}_{L^{2}(\mathbb{R}^{2})} \leq \delta,
\end{align}
where $\delta > 0$ is error level, $\mathcal{K}$ is some suitable admissible set of the approximate solution,
$\tilde{p}(x)$ defined as in (\ref{1defp}).

By applying the Lagrangin formulation, the variable TV restoration model (\ref{1Question}) can be transformed
into the following unconstrained minimization problem:
\begin{align}\label{1_unstranation}
u^{\delta} = \argmin_{u \in \mathcal{K}} \int_{\Omega}|\nabla u|^{\tilde{p}(x)} dx +
\frac{\lambda}{2}\|S u - g^{\delta}\|^{2}_{L^{2}(\mathbb{R}^{2})},
\end{align}
where $\lambda$ is a positive parameter that controls the tradeoff between a good fit to the measurement data and the regularized
solution. For each $\delta > 0$, there exists some $\lambda$ such that (\ref{1Question}) and (\ref{1_unstranation}) are
equivalent.

Comparing (\ref{1_variationImage}) and our model (\ref{1_unstranation}), the forward operator is more complex
than the identity operator.
In order to solve (\ref{1_variationImage}), we can simply take
\begin{align}\label{1_u0}
\tilde{p}(x) = p(g) = P_{M}(|\nabla(G_{\tilde{\delta}} * g)(x)|^2)
\end{align}
where $g$ is the noisy image, and the exponent will not change during our computation,
by doing this the Euler-Lagrange equation will be simpler than
\begin{align}\label{1_u0bian}
\tilde{p}(x) = p(u) = P_{M}(|\nabla(G_{\tilde{\delta}} * u)(x)|^2).
\end{align}
For clarity, we list the Euler-Lagrange equations for (\ref{1_unstranation}) with (\ref{1_u0}) as follows
\begin{align}\label{1Euler1}
\begin{split}
0 = -\nabla \cdot \left( \frac{\nabla u}{|\nabla u|}P_{M}(|\nabla (G_{\tilde{\delta}} * g)|^{2})
|\nabla u|^{P_{M}(|\nabla (G_{\tilde{\delta}} * g)|^{2})-1} \right) + \lambda S^{*}(Su - g^{\delta}),
\end{split}
\end{align}
and for (\ref{1_unstranation}) with (\ref{1_u0bian}) as follows
\begin{align}\label{1Euler2}
\begin{split}
0 = & - G_{\tilde{\delta}} * \nabla \cdot \left( |\nabla u|^{P_{M}(|\nabla (G_{\tilde{\delta}} * u)|^2)}
P'_{M}(|\nabla (G_{\tilde{\delta}} * u)|^2) \cdot 2\nabla (G_{\tilde{\delta}} * u) \right)  \\
& - \nabla \cdot \left( |\nabla u|^{P_{M}(|\nabla (G_{\tilde{\delta}} * u)|^2)} \frac{\nabla u}{|\nabla u|} \right)
+ \lambda S^{*}(Su - g^{\delta}).
\end{split}
\end{align}

For the model (\ref{1_variationImage}), because the edges will not change so much during the computation,
the reduction (\ref{1_u0}) which is taken in \cite{var2009} is suitable.
However, for our model (\ref{1_unstranation}), because the edges will change dramatically during the evolution process,
we must iterate the value of $\tilde{p}(x)$ during our computation.
To make this clear, we consider an image as the initial data and the solution $v(1,x)$ of the fractional diffusion equation
(\ref{1_timespacefrac}) shown in figure \ref{4_2original}.
\begin{figure}[htbp]
\centering
\includegraphics[width=0.9\textwidth]{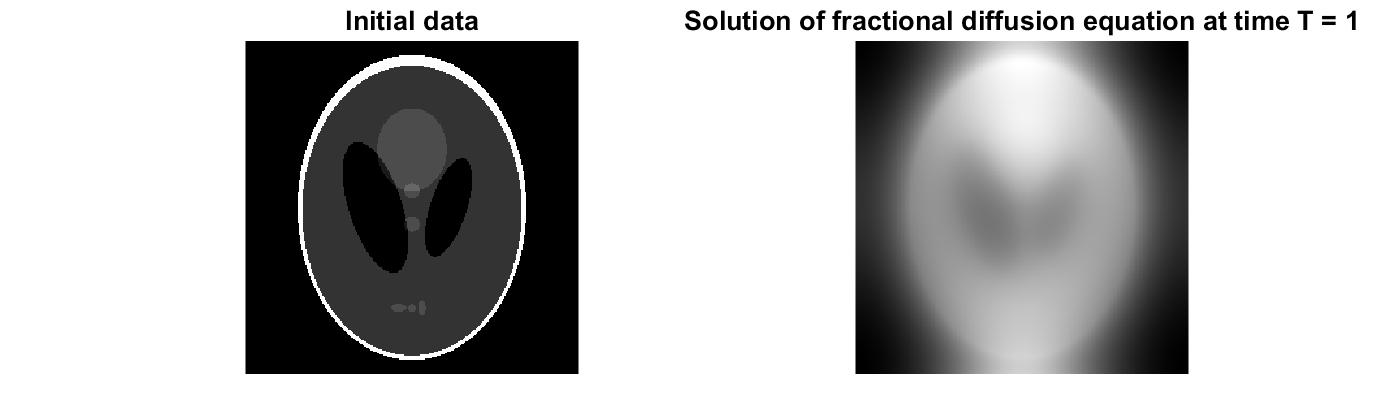}
\caption{Left:~Initial data;
Right:~The solution of the fractional diffusion equation (\ref{1_timespacefrac}) at time $T = 1$ with $\alpha = 0.6$, $\beta = 0.9$.}\label{4_2original}
\end{figure}
The solution $v(1,x)$ shown on the right hand side of figure \ref{4_2original} is calculated by Fourier transform and formula (\ref{1_fourierSolution})
with $\alpha = 0.6$, $\beta = 0.9$ and $T = 1$.
The left image in figure \ref{1_boundary} is the boundary of the initial data detected by Canny algorithm in the Matlab toolbox.
The right image in figure \ref{1_boundary} is the boundary of $v(1,x)$ also detected by Canny algorithm.
From these figures, it is clear that the boundary of the initial data will change dramatically during the fractional
evolution process as claimed in the beginning of this paragraph.
\begin{figure}[htbp]
\centering
\includegraphics[width=0.9\textwidth]{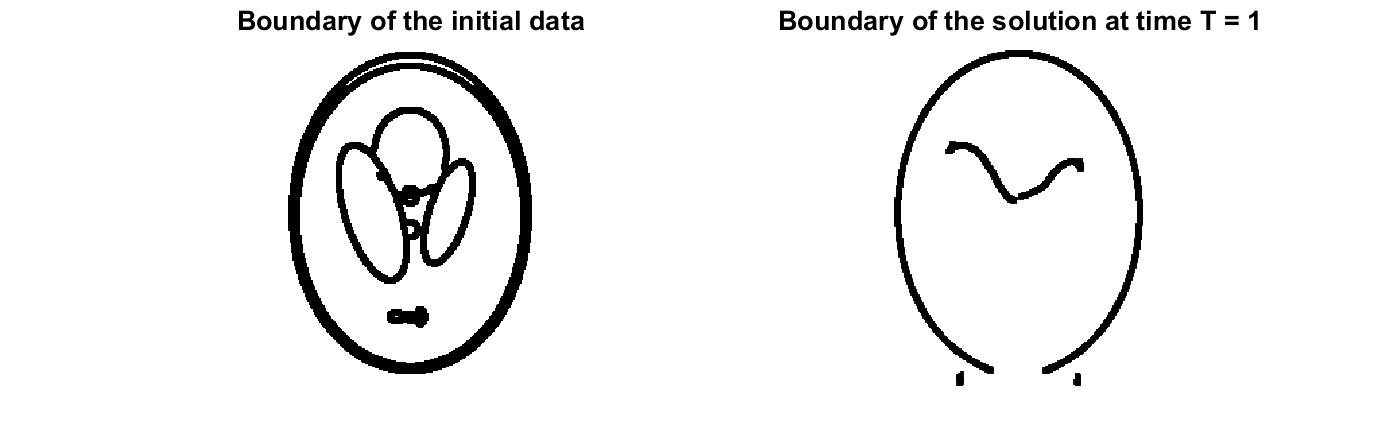}
\caption{Left:~Boundaries of the initial data;
Right:~Boundaries of the solution of the fractional diffusion equation (\ref{1_timespacefrac})
at time $T = 1$ with $\alpha = 0.6$, $\beta = 0.9$.}\label{1_boundary}
\end{figure}

In summary, theories about existence, uniqueness and stability will be proved in a
very general setting, then restricted to backward problem for equation (\ref{1_timespacefrac})
we propose a modified Bregman iterative algorithm to solve problem (\ref{1_unstranation}).
In the following, we will describe the key point of our proof.
In order to prove existence, uniqueness and stability of problem (\ref{1_unstranation}), we first generalize
the theory constructed in \cite{BV1994} to our variable total variation regularization model, during the proof we propose
a concept named CBV-coercive. After building the general theory, we need to verify the operator $S$ appeared in (\ref{1_timespacedomain})
satisfy the conditions in the general theory. One of the new ingredients of this paper is that
we propose an abstract fractional evolution equation (\ref{3_timespacefracAbstract}), then prove the
solution operator of this abstract evolution equation satisfy the required conditions.
Through the abstract formulation, we obtain the existence, uniqueness and stability when
the forward problem is (\ref{1_timespacefrac}) and in addition, all the theoretical results can be applied to
more general systems.
More specifically, our results is valid for system (\ref{3_timespacefracAbstract}) in section 2
with the spatial derivative operator $-A$ generate a $C_{0}$-semigroup and satisfy condition (\ref{2_3judgecondition}).

The second new ingredients of this paper is that we propose a modified Bregman iterative
algorithm to solve problem (\ref{1_unstranation}). To the best of our knowledge,
researchers use Euler-Lagrange equations directly or construct an evolution process based on Euler-Lagrange
equations to solve image restoration problems with variable TV regularizing term.
In the traditional image restoration problem, the edges will not change dramatically for the forward
operator is the identity operator, hence, we can assume (\ref{1_u0}) which highly reduce the computational task.
In our setting, the edges will change during the diffusion process as shown in figure \ref{1_boundary},
so we must iterate the exponent $\tilde{p}$
in our algorithm. Bregman iteration \cite{OsherBurger2005} is an efficient methods used to solve TV regularization based image restoration.
In the framework of Bregman iteration methods, we obtain $\tilde{p}_{n+1}$ by using the value of recovered
image $u_{n}$, so during every iteration we can use the Euler-Lagrange equations as in the (\ref{1_u0}) case.
Hence, on one hand we allow the exponent $\tilde{p}$ update during each iteration.
On the other hand, a simple Euler-Lagrange equation can be used to reduce the computational load.
However, in our modified Bregman iterative algorithm, the regularizing term changed its form at each iteration,
so we need more techniques to provide theoretical analysis of our algorithm. In section 3, we prove the convergence and
provide a practical stopping criterion based on the detailed analysis.

The organization of this paper is as follows. In section 2, we propose the concept of CBV-coercive and
build a general theory then use the general theory to a general linear model with variable TV regularizing term.
By using operator semigroup and fractional operator semigroup theory, we prove the backward problems for an abstract fractional
differential equation satisfying the conditions in our general theory.
In section 3, we propose modified Bregman iterative algorithm, then provide detailed theoretical analysis.
Finally, the numerical implementations are given in section 4 to support our theoretical results and
to show the validity of the proposed algorithm.

\section{Existence, Uniqueness and Stability}
In this section, we will prove existence, uniqueness and stability of our minimization problem (\ref{1_unstranation})
in a general setting.
Here we need to clarify some notations used through all the following parts of this paper.
\begin{itemize}
  \item $d$ stands for dimension; $\Omega \subset \mathbb{R}^{d}$ is a bounded domain with Lipschitz boundary;
  \item $C^{m}$ will stands for functions with continuous derivatives up to order $m$; $C_{c}^{m}$ stands for
  compactly supported function with continuous derivatives up to order $m$;
  \item $W^{m,p}$ is the usual Sobolev space with weak derivatives of order up to $m$ belongs to $L^{p}$;
  For simplicity, we denote $H^{m} := W^{m,p}$ when $p = 2$; $H_{0}^{m}$ will stands for the closure of $C_{c}^{\infty}$
  in $H^{m}$;
  \item For a subset $\Omega \in \mathbb{R}^{d}$, $\chi_{\Omega}$ stands for indicator function which equal to $1$ in $\Omega$
  and equal to $0$ outside of $\Omega$;
  \item $|\Omega|$ stands for the Lebesgue measure of $\Omega \subset \mathbb{R}^{d}$;
  \item If $S$ is a bounded linear operator, we will denote $\|S\|$ as the operator norm of $S$;
  \item BV in this paper stands for functions of bounded variation, the norm defined as
  \begin{align*}
  \|u\|_{BV} := \|u\|_{L^{1}(\Omega)} + \|u\|_{\dot{BV}}
  \end{align*}
  where $\|u\|_{\dot{BV}} := \sup_{\sigma \in V} \int_{\Omega} (-u \mathrm{div}\sigma) dx$
  with $V := \{ \sigma \in C_{c}^{1}(\Omega;\mathbb{R}^{d}) : |\sigma(x)| \leq 1 \text{ for all } x \in \Omega \}$.
\end{itemize}

\subsection{General Theory}
In this subsection, we build a general theory for the following unconstrained minimization problem
\begin{align}\label{2main_problem}
\min_{u \in L^{q}(\Omega)} T(u).
\end{align}

In order to use compactness properties of function spaces for unconstrained minimization problems,
we introduce the following property: define $T$ to be CBV-coercive if
\begin{align}\label{2condition}
T(u) \rightarrow +\infty \quad \text{whenever} \quad J(u) \rightarrow +\infty,
\end{align}
where $J(\cdot)$ satisfies $\|u\|_{BV} \leq C J(u)$.
\begin{theorem}\label{2well-posedness}
Suppose $J$ is defined as in (\ref{2condition}) and $T$ is CBV-coercive. If $1 \leq q < \frac{d}{d-1}$ and $T$ is lower
semi-continuous, then problem (\ref{2main_problem}) has a solution. If in addition $q = \frac{d}{d-1}$, dimension $d \geq 2$,
and $T$ is weakly lower semi-continuous, then a solution also exists.
In either case, the solution is unique if $T$ is strictly convex.
\end{theorem}
\begin{proof}
Let $u_{n}$ be a minimizing sequence for $T$; in other words,
\begin{align}\label{2proof1_1}
\lim_{n \rightarrow \infty} T(u_{n}) = \inf_{u \in L^{q}(\Omega)}T(u) := T_{\text{min}}.
\end{align}
Since $T$ is CBV-coercive, the $\{u_{n}\}$ are BV-bounded.
By Theorem 2.5 in \cite{BV1994}, there exists a subsequence $u_{n_{k}}$ which converges to some $\bar{u} \in L^{q}(\Omega)$.
Convergence is weak if $q = \frac{d}{d-1}$. By the (weak) lower semi-continuity of $T$,
\begin{align}\label{2proof1_2}
T(\bar{u}) \leq \liminf_{k \rightarrow \infty} T(u_{n_{k}}) = T_{\text{min}}.
\end{align}
Uniqueness of minimizers follows immediately from strict convexity.
\end{proof}

Next, we consider a sequence of perturbed problems
\begin{align}\label{2perturbed}
\min_{u \in L^{q}(\Omega)}T_{n}(u).
\end{align}
\begin{theorem}\label{2perturbed_th}
Assume $J$ is defined as in (\ref{2condition}), $1 \leq q < \frac{d}{d-1}$ and that $T$ and each of the $T_{n}$s are
CBV-coercive, lower semi-continuous, and have a unique minimizer. Assume in addition:
\begin{enumerate}
  \item Uniform CBV-coercivity: for any sequence $v_{n} \in L^{q}(\Omega)$,
  \begin{align}\label{2stability1}
  \lim_{n \rightarrow \infty}T_{n}(v_{n}) = \infty \quad \text{whenever} \quad \lim_{n\rightarrow \infty}J(v_{n}) = \infty.
  \end{align}
  \item Consistency: $T_{n} \rightarrow T$ uniformly on J-bounded sets, i.e. given $B > 0$ and $\epsilon > 0$,
  there exists $N$ such that
  \begin{align}\label{2stability2}
  |T_{n}(u) - T(u)| < \epsilon \quad \text{whenever} \quad n \geq N, \,\, J(u) \leq B.
  \end{align}
\end{enumerate}
Then problem (\ref{2main_problem}) is stable with respect to the perturbations (\ref{2perturbed}), i.e. if
$\bar{u}$ minimizes $T$ and $u_{n}$ minimizes $T_{n}$, then
\begin{align*}
\|u_{n} - \bar{u}\|_{L^{q}(\Omega)} \rightarrow 0.
\end{align*}
If $q = \frac{d}{d-1}$, $d \geq 2$, and one can replaces the lower semi-continuity assumption on $T$ and each $T_{n}$
by weak lower semi-continuity, then convergence is weak:
\begin{align*}
u_{n} - \bar{u} \rightharpoonup 0.
\end{align*}
\end{theorem}
\begin{proof}
Note that $T_{n}(u_{n}) \leq T_{n}(\bar{u})$, by assumption (2), we have
\begin{align*}
\liminf_{n \rightarrow \infty}T_{n}(u_{n}) \leq \limsup_{n\rightarrow \infty}T_{n}(u_{n}) \leq T(\bar{u}) < \infty
\end{align*}
and hence by assumption (1), the $u_{n}$s are $J$-bounded.
Remember the properties of $J$, the $u_{n}$s are BV-bounded.
Now suppose our results does not hold. By Theorem 2.5 in \cite{BV1994}, there exists
a subsequence $u_{n_{k}}$ which converges in $L^{q}(\Omega)$(weak $L^{q}(\Omega)$ if $q = \frac{d}{d-1}$)
to some $\hat{u} \neq \bar{u}$. By the (weak) lower semi-continuity of $T$,
\begin{align*}
T(\hat{u}) & \leq \liminf_{k \rightarrow \infty}T(u_{n_{k}}) \\
& = \lim_{k \rightarrow \infty}(T(u_{n_{k}}) - T_{n_k}(u_{n_k})) + \liminf_{k \rightarrow \infty}T_{n_k}(u_{n_k})   \\
& \leq T(\bar{u}).
\end{align*}
But this contradicts the uniqueness of the minimizer $\bar{u}$ of $T$.
\end{proof}

\subsection{Variable TV Regularization for General Linear Problems}

In this subsection, we consider the following special form of $T$:
\begin{align}\label{2_2pro}
T(u) = F(u) + \frac{\lambda}{2} \|Su - g\|_{L^{2}(\Omega)}^{2},
\end{align}
where
\begin{align}\label{2_2defF}
F(u) = \max_{\nu(x) \geq 0, |\sigma(x)| \leq 1} \int_{\Omega} \left( \nabla \cdot (\tilde{p}\nu^{\tilde{p} - 1}\sigma)u
- (\tilde{p} - 1)\nu^{\tilde{p}} \right)dx
\end{align}
and $\nu$ ranges over the set of $C^{1}(\bar{\Omega})$ functions with positive minimum, $\sigma$
ranges over functions in $C_{c}^{1}(\Omega)$ with $|\sigma(x)| \leq 1$, $\tilde{p}$ defined as in (\ref{1defp}),
$S$ is a linear bounded operator from $L^{q}(\Omega)$ to $L^{2}(\Omega)$, $g$ is a function in $L^{2}(\Omega)$, $\lambda > 0$ is
a positive real number.
As demonstrated in \cite{var2009}, if $u, \tilde{p}$ are $C^1$, $F(u)$ defined above is equivalent to
\begin{align}\label{2_2defst}
\int_{\Omega} |\nabla u|^{\tilde{p}} dx.
\end{align}
So instead of (\ref{1_unstranation}) in the previous section, in this subsection we consider $T$ defined in (\ref{2_2pro}).
For this particular $T$, we define
\begin{align}\label{2_2defJ}
J(u) = \|u\|_{L^{2}(\Omega)} + F(u) + \frac{1}{4}|\Omega|,
\end{align}
which obviously satisfies $\|u\|_{BV} \leq C J(u)$.
\begin{theorem}\label{2_2mainth}
Assume $d = 2$ and that $1 \leq q \leq \frac{d}{d-1}$, $S$ is a linear bounded operator from $L^{q}(\Omega)$ to $L^{2}(\Omega)$ and $g$ is a function in $L^{2}(\Omega)$. In addition, we assume that
\begin{align}\label{2_2assumeS}
S\chi_{\Omega} \neq 0.
\end{align}
Then $T$ defined in (\ref{2_2pro}) is CBV-coercive with $J$ defined in (\ref{2_2defJ}), and the functional $T$ has a minimizer.
\end{theorem}
\begin{proof}
The lower semi-continuous or weakly lower semi-continuous of $F(u)$ follows from Theorem 9 and Theorem 12 in \cite{var2009}.
Hence the (weakly) lower semi-continuous of $J(u)$ obviously hold.
If we can prove $T$ in (\ref{2_2pro}) is CBV-coercive, by Theorem \ref{2well-posedness}, the functional $T$
has a minimizer. So the main task is to prove $T$ is CBV-coercive.
Decompose $u$ as follows:
\begin{align}\label{2_2decom}
u = \nu + w
\end{align}
where
\begin{align}\label{2_2decom1}
w = \frac{\int_{\Omega}u dx}{|\Omega|}\chi_{\Omega} \quad\quad \int_{\Omega}\nu dx = 0.
\end{align}
Using Poincar\'{e} inequality and H\"{o}lder's inequality, there exists a positive constant $C$ such that
for any $q$ such that $1 \leq q \leq \frac{d}{d-1} = r$,
\begin{align}\label{2_2proof1}
\begin{split}
\|\nu\|_{L^{q}(\Omega)} & \leq |\Omega|^{\frac{1}{q} - \frac{1}{r}}\|\nu\|_{L^{r}(\Omega)}  \\
& \leq (|\Omega| + 1)^{1-\frac{1}{r}}C \|\nu\|_{\dot{BV}}   \\
& \leq C_{1} \left(F(\nu) + \frac{1}{4}|\Omega|\right)
\end{split}
\end{align}
where $C_{1} := (|\Omega| + 1)^{\frac{1}{d}}C$. In the last inequality of (\ref{2_2proof1}),
we used (35) in \cite{var2009}.
Using (\ref{2_2proof1}) and the decomposition (\ref{2_2decom}), we have
\begin{align}\label{2_2proof2}
\begin{split}
J(u) & = \|u\|_{L^{2}(\Omega)} + F(u) + \frac{1}{4} |\Omega|    \\
& \leq \|w\|_{L^{2}(\Omega)} + (C_{1} + 1)(F(\nu) + \frac{1}{4} |\Omega|).
\end{split}
\end{align}
From the assumption (\ref{2_2assumeS}), there exists $C_{2} > 0$ such that
\begin{align}\label{2_2proof3}
\|Sw\|_{L^{2}(\Omega)} = C_{2}\|w\|_{L^{2}(\Omega)}.
\end{align}
From the definition of $T$ and the decomposition (\ref{2_2decom}), we obtain
\begin{align}\label{2_2proof4}
\begin{split}
T(u) & = F(u) + \frac{\lambda}{2} \|(S\nu - g) + Sw\|_{L^{2}}^{2}   \\
& \geq F(u) + \frac{\lambda}{2} (\|S\nu - g\|_{L^{2}} - \|Sw\|_{L^{2}})^{2}   \\
& \geq F(u) + \frac{\lambda}{2} \|Sw\|_{L^{2}} (\|Sw\|_{L^{2}} - 2\|S\nu - g\|_{L^{2}}).
\end{split}
\end{align}
By (\ref{2_2proof1}), we obtain
\begin{align}\label{2_2proof5}
\|S\nu - g\|_{L^{2}} \leq C_{1}\|S\|F(u) + \frac{1}{4}C_{1}\|S\| |\Omega| + \|g\|_{L^{2}}.
\end{align}
Combining (\ref{2_2proof3}), (\ref{2_2proof4}) and (\ref{2_2proof5}), we have
\begin{align}\label{2_2proof6}
\begin{split}
T(u) \geq & F(u) + \frac{\lambda}{2}C_{2}\|w\|_{L^{2}}(C_{2}\|w\|_{L^{2}}  \\
& \quad\quad - 2(C_{1}\|S\|F(u) + \frac{1}{4}C_{1}\|S\| |\Omega| + \|g\|_{L^{2}})).
\end{split}
\end{align}
From the definition of $T(u)$ in (\ref{2_2pro}), we obtain that
\begin{align}\label{2_2proof8}
F(u) \leq T(u).
\end{align}
\begin{description}
  \item[Case 1] $C_{2}\|w\|_{L^{2}} - 2(C_{1}\|S\|F(u) + \frac{1}{4}C_{1}\|S\| |\Omega| + \|g\|_{L^{2}})) \geq 1.$
  From (\ref{2_2proof6}), we obtain
  \begin{align}\label{2_2proof7}
  \|w\|_{L^{2}} \leq \frac{2}{\lambda C_{2}}T(u).
  \end{align}
  Hence, considering (\ref{2_2proof8}), we finally obtain
  \begin{align}\label{2_2proof9}
  J(u) \leq \left( \frac{2}{\lambda C_{2}} + (C_{1}+1) \right)T(u) + \frac{C_{1}+1}{4}|\Omega|.
  \end{align}
  \item[Case 2] $C_{2}\|w\|_{L^{2}} - 2(C_{1}\|S\|F(u) + \frac{1}{4}C_{1}\|S\| |\Omega| + \|g\|_{L^{2}}) < 1.$
  Obviously, we have
  \begin{align}\label{2_2proof10}
  \|w\|_{L^{2}} \leq \frac{1 + 2(C_{1}\|S\|F(u) + \frac{1}{4}C_{1}\|S\| |\Omega| + \|g\|_{L^{2}})}{C_{2}}
  \end{align}
  Combining (\ref{2_2proof8}) with the above inequality, we obtain
  \begin{align}\label{2_2proof11}
  J(u) \leq \left( \frac{2C_{1}\|S\|}{C_{2}} + C_{1} + 1 \right)T(u) + \frac{1+\frac{1}{4}C_{1}\|S\| |\Omega| +\|g\|_{L^{2}}}{C_{2}}.
  \end{align}
\end{description}
Considering case 1 and case 2, (\ref{2_2proof9}) and (\ref{2_2proof11}) yields the CBV-coercivity of $T$.
\end{proof}
\begin{remark}\label{2_2remark1}
In the above theorem, if we assume the solution lies in a closed convex subset $\mathcal{K}$ of $L^{q}$ with $1\leq q \leq \frac{d}{d-1}$,
$T$ can be written as
\begin{align}\label{2_2remarkin1}
T(u) = F(u) + \frac{\lambda}{2} \|Su - g\|_{L^{2}(\Omega)}^{2} + \frac{1}{\chi_{\mathcal{K}}}.
\end{align}
Because the above $T$ is strictly convex, using same procedure as in the proof of theorem \ref{2_2mainth}, we obtain
the following theorem.
\begin{theorem}\label{2_2mainthunique}
Assume $d = 2$ and that $1 \leq q \leq \frac{d}{d-1}$, $\mathcal{K}$ is a closed convex subset of $L^{q}(\Omega)$,
$S$ is a linear bounded operator from $L^{q}(\Omega)$ to $L^{2}(\Omega)$ and $g$ is a function in $L^{2}(\Omega)$. In addition, we assume that
\begin{align}\label{2_2assumeSS}
S\chi_{\Omega} \neq 0.
\end{align}
Then $T$ defined in (\ref{2_2pro}) is CBV-coercive with $J$ defined in (\ref{2_2defJ}), and the functional $T$ has a
unique constrained minimizer over $\mathcal{K}$.
\end{theorem}
\end{remark}

Next, we addresses the stability of minimizers to functionals of (\ref{2_2pro}).
Consider perturbed functionals
\begin{align}\label{2_2proper}
T_{n}(u) = F(u) + \frac{\lambda}{2} \|S_{n}u - g_{n}\|_{L^{2}(\Omega)}^{2},
\end{align}
\begin{theorem}\label{2_2mainper}
Assume $1 \leq q \leq \frac{d}{d-1}$, $\lim_{n\rightarrow \infty}\|g_{n} - g\|_{L^{2}(\Omega)} = 0$,
the $S_{n}$s are each bounded linear and converge pointwise to $S$, and for each $n$,
\begin{align}\label{2_2per1}
\|S_{n}\chi_{\Omega}\|_{L^{2}} \geq \gamma > 0.
\end{align}
Also assume each $S_{n}$ has a unique minimizer $u_{n}$ and that $S$ has a unique minimizer $\bar{u}$.
Then for $1 \leq q < \frac{d}{d-1}$, we have
\begin{align}\label{2_2per2}
\lim_{n\rightarrow \infty}\|u_{n} - \bar{u}\|_{L^{q}(\Omega)} = 0,
\end{align}
for $q = \frac{d}{d-1}$, the convergence is weak
\begin{align}\label{2_2per3}
u_{n} \rightharpoonup \bar{u}.
\end{align}
\end{theorem}
\begin{proof}
It suffices to show that conditions (1) and (2) of Theorem \ref{2perturbed_th} hold.
For condition (1), put $u_{n} = \nu_{n} + w_{n}$ as in (\ref{2_2decom}) and (\ref{2_2decom1}),
and repeat the proof of theorem \ref{2perturbed_th}.
Since $\|S_{n}\chi_{\Omega}\|_{L^{2}} \geq \gamma\sqrt{|\Omega|}\|w_{n}\|_{L^{2}(\Omega)}$,
letting $M$ be an upper bound on $S$ and each $S_{n}$, $m$ is an upper bound on $\|g\|_{L^{2}}$
and each $\|g_{n}\|_{L^{2}}$, one obtains
\begin{align}\label{2_2per4}
\begin{split}
T_{n}(u_{n}) \geq & F(u_{n}) + \frac{\lambda}{2}\gamma \sqrt{|\Omega|}\|w_{n}\|_{L^{2}}(\gamma \sqrt{|\Omega|}\|w_{n}\|_{L^{2}}  \\
& \quad\quad - 2(C_{1}M F(u_{n}) + \frac{1}{4}C_{1}\|S_{n}\|\text{Vol}(\Omega) + m)).
\end{split}
\end{align}
This yields uniform CBV-coercivity by same argument as in the proof of Theorem \ref{2perturbed_th}.
Since
\begin{align*}
|T_{n}(u) - T(u)| \leq & \frac{\lambda}{2}(\|S_{n}u - Su\|_{L^{2}} + \|g_{n} - g\|_{L^{2}})
((\|S_{n}\|+\|S\|)\|u\|_{L^{2}}     \\
& + \|g_{n}\|_{L^{2}} + \|g\|_{L^{2}}),
\end{align*}
condition (2) is obviously satisfied.
\end{proof}

\subsection{Variable TV Regularization for Fractional Backward Diffusion}
In this subsection, we firstly construct an abstract fractional evolution equation based on the following
time-space fractional diffusion system for homogeneous media:
\begin{align}\label{3_timespacefrac}
\begin{split}
\left\{ \begin{array}{ll}
\partial_{t}^{\alpha} v(t,x) + (-\Delta)^{\beta}v(t,x) = 0 & (x,t) \in \mathbb{R}^{2} \times (0,\infty) \\
v(x,0) = u(x) & x \in \mathbb{R}^{2}
\end{array} \right. .
\end{split}
\end{align}
with $0 < \alpha \leq 1$, $\beta \in (1/2,1]$.
$u$ is the initial data in $\mathbb{R}^{2}$ with compact support in a bounded convex open subset $\Omega$
of the plane with Lipschitz continuous boundary $\partial\Omega$.

Define an operator $A = -\Delta$ with domain
\begin{align}\label{2_3domainA}
D(A) = \{ u \in H_{0}^{1}(\mathbb{R}^{2}) : \Delta u \in L^{2}(\mathbb{R}^{2}) \}.
\end{align}
Let $X = L^{2}(\mathbb{R}^{2})$, then $A$ is a bounded linear operator
defined on $X$. By Theorem 2.4.1 in \cite{fracope}, we know that $A$ is m-accretive and
$-A$ generates a uniformly exponentially stable and contractive $C_{0}$-semigroup, denoted as $T(t)$.
From the proof of Theorem 2.4.1 in \cite{fracope}, we know that there exists a constant $c > 0$ such that
\begin{align}\label{2_3decay}
\|T(t)\| \leq e^{-ct},
\end{align}
where $\|\cdot\|$ is the operator norm.
Considering (\ref{2_3decay}), for $\frac{1}{2} < \beta \leq 1$, we can define the following bounded linear operator
\begin{align}\label{2_3nagativePower}
A^{-\beta}x = \frac{1}{\Gamma(\beta)} \int_{0}^{\infty} t^{\beta - 1}T(t)x dt \quad \forall \,\, x\in X.
\end{align}
Then as illustrated in \cite{fracope}, \cite{Haase2006} or \cite{wang2006}, we can define the operator $A^{\beta}$ as
the inverse operator of $A^{-\beta}$. Hence, we have
\begin{align}\label{2_3opsitivePower}
A^{\beta}A^{-\beta}x = x \quad \forall \,\, x\in X \quad \text{and} \quad D(A^{\beta}) = R(A^{-\beta}).
\end{align}
For any $x \in D(A)$, the above positive power of operator $A$ has the following Balakrichnan representations \cite{Haase2006}
\begin{align}\label{2_3formpositive}
A^{\beta}x = \frac{\sin(\pi \alpha)}{\pi} \int_{0}^{\infty} t^{\alpha-1} (t+A)^{-1} Ax dx \quad \forall \,\, x\in D(A),
\end{align}
which may provide more intuitive ideas to the readers.
With these preparations, we can recast system (\ref{3_timespacefrac}) into the following abstract ordinary differential
equations on Banach space $X = L^{2}(\mathbb{R}^{2})$ as follows
\begin{align}\label{3_timespacefracAbstract}
\begin{split}
\left\{ \begin{array}{ll}
\partial_{t}^{\alpha} v(t) + A^{\beta} v(t) = 0 & t \in (0,\infty) \\
v(0) = u &
\end{array} \right. .
\end{split}
\end{align}
\begin{remark}\label{2_3equivalence}
Usually the fractional Laplacian operator defined as $$(-\Delta)^{\beta}v := \mathcal{F}^{-1}(|\xi|^{2\beta}\hat{v}).$$
So in order to obtain a meaningful abstract form (\ref{3_timespacefracAbstract}), we need to state the equivalence of
$A^{\beta}$ defined in (\ref{2_3formpositive}) and the usual definition by fourier transform.
By Proposition 8.3.3 in \cite{Haase2006}, we know that if $v \in D(A^{\beta})$ then
$A^{\beta}v = \mathcal{F}^{-1}(|\xi|^{2\beta}\hat{v})$.
That is to say if $v \in H_{0}^{1}(\mathbb{R}^{2}) \cap H^{2\beta}(\mathbb{R}^{2})$,
the abstract form is equivalent to the usual definition by fourier transform.
\end{remark}

\begin{remark}\label{jieshi}
The operator $A$ can be defined more generally as follow
\begin{align*}
A = -\nabla\cdot (a(x)\nabla \cdot) + c(x),
\end{align*}
where $a(\cdot), c(\cdot)$ are functions in $C^{1}$ and in addition,
we suppose the operator $A$ defined above satisfies the strong elliptic condition.
It is well known that this operator generate a contractive $C_{0}-$semigroup \cite{evans} and we can define
fractional operator as in (\ref{2_3nagativePower}),(\ref{2_3opsitivePower}) and (\ref{2_3formpositive}),  so
the abstract fractional evolution equation (\ref{3_timespacefracAbstract}) incorporate
a natural generalization of fractional Laplace operator.
In the following part of this article, we only present the proof of the Laplace case.
That is because once we define the general operator mentioned in this remark appropriately, it
satisfies all the properties of the operator semigroup which we used and the proof will be almost same.
\end{remark}

Now we can prove the main theorem in this subsection.
\begin{theorem}\label{2_3mainth}
Let $\mathcal{K}$ be a closed convex subset of $L^{2}(\Omega)$, $g^{\delta}$ is the measurement data
in time $T$. Then the optimization problem
\begin{align}\label{2_3Question}
\argmin_{u \in \mathcal{K}} F(u) + \frac{\lambda}{2}\|S u - g^{\delta}\|_{L^{2}(\mathbb{R}^{2})}
\end{align}
has a unique minimizer over $\mathcal{K}$ for any fixed $\lambda > 0$.
\end{theorem}
\begin{proof}
We attempt to use Theorem \ref{2_2mainthunique} to obtain the result, so we need to verify that the operator
$S$ is bounded from $L^{2}$ to $L^{2}$, and $S \chi_{\Omega} \neq 0$.

\textbf{Step 1}. Bounded of operator $S$.
Taking the measure $\mu$ in \cite{Jia2014} to be $\delta_{\beta}(\cdot)$, using Theorem 3.7 in \cite{Jia2014}, we find that
$-A^{\beta}$ generates a bounded $C_{0}$-semigroup denoted as $S_{\beta}(t)$.
For simplicity, we denote $S_{\beta}$ as $S_{\beta}(T)$ for short.
Using Corollary 2.10 in \cite{Bajlekova2001}, we know that $-A^{\beta}$ generate an $\alpha$-order
fractional semigroup proposed in \cite{Peng2012}. We denote the $\alpha$-order fractional
semigroup generated by $-A^{\beta}$ as $S_{\alpha,\beta}(t)$, particularly for $t = T$, denote
$S_{\alpha,\beta}$ as $S_{\alpha,\beta}(T)$ for short. Instead of the semigroup property, this $\alpha$-order fractional
semigroup satisfies the following equality
\begin{align*}
\int_{0}^{t+s}\frac{S_{\alpha,\beta}(\tau)}{(t+s-\tau)^{\alpha}} d\tau
- \int_{0}^{t}\frac{S_{\alpha,\beta}(\tau)}{(t+s-\tau)^{\alpha}} d\tau
& - \int_{0}^{s}\frac{S_{\alpha,\beta}(\tau)}{(t+s-\tau)^{\alpha}} d\tau  \\
& = \alpha \int_{0}^{t}\int_{0}^{s} \frac{S_{\alpha,\beta}(\tau_{1})
S_{\alpha,\beta}(\tau_{2})}{(t+s-\tau_{1}-\tau_{2})^{1+\alpha}}d\tau_{1}d\tau_{2}
\end{align*}
in the strong operator topology.
Noting our definition of $v(x,1)$ for $x \notin \Omega$ in Section 1, we have the following expression for operator $S$:
\begin{align}\label{2_3defofS}
\begin{split}
(Su)(x) = \left\{ \begin{array}{ll}
(S_{\alpha,\beta}u)(x), & x \in \Omega \\
g^{\delta}(x), & x \notin \Omega.
\end{array} \right. .
\end{split}
\end{align}
Now the meaning of the operator $S$ is not restricted to the one used in (\ref{1_timespacedomain}).
In order to obtain our results, we need to introduce
the following lemma (restated in our setting) proved in Section 3 of \cite{Bajlekova2001}.
\begin{lemma}\label{2_3subordinate}
Let $S_{\beta}(t)$ is a bounded $C_{0}$-semigroup, $S_{\alpha,\beta}(t)$ is an $\alpha$-order fractional
semigroup with $\alpha \in (0,1)$. Then the following representation holds
\begin{align}\label{2_3sub1}
S_{\alpha,\beta}(t) = \int_{0}^{\infty}\phi_{t,\alpha}(s) S_{\beta}(s) ds, \quad t > 0,
\end{align}
where $\phi_{t,\alpha}(s) = t^{-\alpha}\Phi_{\alpha}(st^{-\alpha})$, $\Phi_{\alpha}(z)$ is the function of Wright type defined as
\begin{align}\label{2_3defofPhi}
\Phi_{\alpha}(z) := \sum_{n=0}^{\infty} \frac{(-z)^{n}}{n! \Gamma(-\alpha n + 1 - \alpha)}, \quad 0< \alpha < 1,
\end{align}
and (\ref{2_3sub1}) holds in the strong sense.
\end{lemma}
For functions of Wright type, there are many usefully literatures \cite{Wright1940,Mainardi1996,Mainardi1994}.
Here we only recall that $\Phi_{\alpha}(t)$ is a probability density function satisfies:
\begin{align}\label{2_3proPhi}
\Phi_{\alpha}(t) \geq 0,\quad t>0; \quad \int_{0}^{\infty}\Phi_{\alpha}(t)dt = 1.
\end{align}
Considering Lemma \ref{2_3subordinate} and the above properties (\ref{2_3proPhi}), for every $u \in X$, we have
\begin{align}\label{2_3guji}
\begin{split}
\|Su\|_{L^{2}(\Omega)} = \|S_{\alpha,\beta}u\|_{L^{2}(\Omega)} & \leq \int_{0}^{\infty} \phi_{T,\alpha}(s)\|S_{\beta}(s)u\|_{L^{2}} ds \\
& \leq C \int_{0}^{\infty} \phi_{T,\alpha}(s) ds \|u\|_{L^{2}}  \\
& = C \|u\|_{L^{2}(\Omega)},
\end{split}
\end{align}
where we used the fact that $S_{\beta}(s)$ is a bounded $C_{0}$-semigroup.

\textbf{Step 2}. Operator $S$ does not annihilate constant functions.
Assume $\|S_{\alpha,\beta} \chi_{\Omega}\|_{L^{2}(\Omega)} = 0$, that is
\begin{align}\label{2_3well2}
\int_{0}^{\infty}\phi_{T,\alpha}(s)S_{\beta}(s)\chi_{\Omega}ds = 0 \quad a.e.
\end{align}
By the definition of $\phi_{T,\alpha}(s)$, we know that
\begin{align}\label{2_3well3}
\int_{0}^{\infty} T^{-\alpha}\Phi_{\alpha}(T^{-\alpha}s) S_{\beta}(s)\chi_{\Omega}ds = 0 \quad a.e.
\end{align}
Since $\Phi_{\alpha}(s)$ is a probability density, if $S_{\beta}(s)\chi_{\Omega} \geq 0$ a.e., we can conclude that
$S_{\beta}(s)\chi_{\Omega} = 0$ a.e. for almost all $s > 0$. By the strong continuity of $C_{0}$-semigroup, we will obtain
$\lim_{s\rightarrow 0}\|S_{\beta}(s)\chi_{\Omega} - \chi_{\Omega}\|_{L^{2}(\mathbb{R}^{2})} = 0$, that is to say
$\|\chi_{\Omega}\|_{L^{2}(\mathbb{R}^{2})} = 0$ which is a contradiction.
So $\|S \chi_{\Omega}\|_{L^{2}(\Omega)} = \|S_{\alpha,\beta} \chi_{\Omega}\|_{L^{2}(\Omega)} > 0$ which
means $S\chi_{\Omega} \neq 0$ in $L^{2}(\Omega)$.

Now we verify the positive condition $S_{\beta}(s)\chi_{\Omega} \geq 0$ a.e..
Since $\chi_{\Omega} \geq 0$, we just need to verify that $S_{\beta}(s)$ is a positive
preserving operator semigroup. For more properties of positive preserving operator semigroup, we
refer to \cite{Jacob1,Arendt2011}. Here, for the reader's convenience, we give the following lemma
(Theorem 4.6.14 in \cite{Jacob1})
which is useful for our proof.
\begin{lemma}\label{2_3judgepositive}
Let $\{T(t)\}_{t\geq0}$ be a strongly continuous contraction semigroup on the space $L^{p}(\mathbb{R}^{d},\mathbb{R})$, $1 < p < \infty$,
with generator $(-A, D(A))$. The semigroup $\{T(t)\}_{t\geq0}$ is positive preserving if and only if
\begin{align}\label{2_3judgecondition}
\int_{\mathbb{R}^{d}}(-Au)(u^{+})^{p-1}dx \leq 0
\end{align}
holds for all $u \in D(A)$.
\end{lemma}
In our setting, $A = -\Delta$, $p = 2$, (\ref{2_3judgecondition}) can be verified as follows
\begin{align}\label{2_3judgeconditionooo}
\int_{\mathbb{R}^{d}}(\Delta u)u^{+}dx & = - \int_{\mathbb{R}^{d}} \nabla u \cdot \nabla u^{+}dx \\
& = - \int_{\mathbb{R}^{d}} |\nabla u^{+}|^{2}dx \leq 0,
\end{align}
where we used
\begin{align*}
\nabla u^{+} = \left\{ \begin{array}{ll}
\nabla u, & \text{a.e. on }\{ u>0 \}  \\
0, & \text{a.e. on }\{ u\leq 0 \}.
\end{array} \right.
\end{align*}
Hence, the semigroup stated in (\ref{2_3decay}) is positive preserving.
Then from the proof of Theorem 3.7 in \cite{Jia2014}, there exists a probability measures $\{\mu(t)\}_{t\geq 0}$ such that
\begin{align}\label{2_3judgeproof1}
S_{\beta}(t) = \int_{0}^{\infty} T(s)\mu(t)ds.
\end{align}
Considering (\ref{2_3judgeproof1}), the $C_{0}$-semigroup $S_{\beta}(t)$ is obviously positive preserving
that is $S_{\beta}(t)\chi_{\Omega} \geq 0$ a.e..
\end{proof}

At last, let us consider the stability of the minimizer of (\ref{2_3Question}) with respect to
the perturbations on the operator $S$ and $g^{\delta}$ for any fixed $\lambda > 0$. We state this result as follows.
\begin{theorem}\label{2_3perturbationTheorem}
Let the hypothesis in Theorem \ref{2_3mainth} be satisfied. Define two functionals
\begin{align*}
& T_{*}(u) := F(u) + \frac{\lambda}{2} \|Su - g\|_{L^{2}(\mathbb{R}^{2})}^{2},  \\
& T_{n}(u) := F(u) + \frac{\lambda}{2} \|S_{n}u - g_{n}\|_{L^{2}(\mathbb{R}^{2})}^{2}, \quad n = 1,2,\ldots
\end{align*}
and denote by
\begin{align*}
u_{*} := \argmin_{u \in \mathcal{K}}T_{*}(u), \quad u_{n} := \argmin_{u \in \mathcal{K}}T_{n}(u).
\end{align*}
If $\lim_{n\rightarrow \infty}\|g_{n} - g\|_{L^{2}(\mathbb{R}^{2})} = 0$,
$S_{n}$ converge pointwise to $S$ as $n \rightarrow \infty$,
then it follows that $u_{n} \rightharpoonup u$ as $n \rightarrow \infty$ in the weak topology of $L^{2}(\Omega)$.
\end{theorem}
\begin{proof}
We will use Theorem \ref{2_2mainper} to obtain the above result. Here we just need to verify condition (\ref{2_2per1}).
From the proof of Theorem \ref{2_3mainth}, we know that there exists a constant $c > 0$ such that
\begin{align}\label{2_3per1}
\|S\chi_{\Omega}\|_{L^{2}} \geq c > 0.
\end{align}
Since here we just concern the case with large enough $n$ and $S_{n}$ converge to $S$ pointwise, we can find a
positive constant $N>0$ such that if $n \geq N$ then we have
\begin{align}\label{2_3per2}
\|S_{n}\chi_{\Omega} - S\chi_{\Omega}\|_{L^{2}} \leq \frac{c}{2}.
\end{align}
From (\ref{2_3per1}) and (\ref{2_3per2}), we obtain
\begin{align}\label{2_3per3}
\|S_{n}\chi_{\Omega}\|_{L^{2}} \geq \frac{c}{2} > 0
\end{align}
for every $n \geq N$.
\end{proof}

\begin{remark}\label{2_3remarkGeneral}
For the general operator $A$ mentioned in remark \ref{jieshi}, in order to obtain the same results as in the Laplace case,
we may need to assume the function $c(\cdot)$ has a small positive bound to verify the conditions mentioned in lemma \ref{2_3judgepositive}.
In summary, the proof of Theorem \ref{2_3mainth} and Theorem \ref{2_3perturbationTheorem} just require
$-A$ generate a contraction $C_{0}$-semigroup and satisfies condition (\ref{2_3judgecondition}).
In addition, by some modifications, we may obtain similar results for space distributed order fractional
diffusion equations studied in \cite{Jia2014,JinBangti2015}. Because the modifications is not trivial, it may
be more appropriate to report in another paper.
\end{remark}

\section{Numerical Approach}
The well-posedness of our optimization problem (\ref{2_3Question}) is obtained in the previous section.
In this section, we propose modified Bregman iterative method to solve (\ref{2_3Question}) efficiently.
The Bregman distance associated with a convex functional $F(\cdot)$ between points $u$ and $v$ is defined as
\begin{align}\label{3_BregmanDistance}
D_{F}^{p}(u,v) := F(u) - F(v) - <p, u - v>,
\end{align}
where $p \in \partial F(v) = \{ w : F(u) - F(v) \geq <w, u - v>, \,\, \forall u \}$ is the sub-gradient of $F(\cdot)$
at the point $v$. Then our optimization problem (\ref{2_3Question}) becomes
\begin{align}\label{3_OptimizationDistance}
u_{*} := \argmin_{u \in \mathcal{K}} F(u_{*}) + <p, u - u_{*}> + D_{F}^{p}(u, u_{*}) +
\frac{\lambda}{2}\|Su - g^{\delta}\|_{L^{2}(\mathbb{R}^{2})}^{2}
\end{align}
with $p \in \partial F(u_{*})$.
Since the forward problem can be solved efficiently in the frequency domain, by Parseval identity, we have the
following equivalent form
\begin{align}\label{3_FreOptimizationDistance}
u_{*} := \argmin_{u \in \mathcal{K}} F(u_{*}) + <p, u - u_{*}> + D_{F}^{p}(u, u_{*}) +
\frac{\lambda}{2}\|\hat{S}\hat{u} - \hat{g}^{\delta}\|_{L^{2}(\mathbb{R}^{2})}^{2}
\end{align}
where $\hat{S}$ defined as in (\ref{1_fourierSolution}).

Instead of solving (\ref{1Question}), Osher et al.\cite{OsherBurger2005} proposed Bregman iterative
regularization to solve (\ref{3_FreOptimizationDistance}) approximately by using the following iterative formula:
\begin{align}\label{3_ApproBregman}
u^{m+1} = \argmin_{u} D^{p^{m}}_{F}(u,u^{m}) + \frac{\lambda}{2} \|\hat{S}\hat{u} - \hat{g}^{\delta}\|_{L^{2}(\mathbb{R}^{2})}^{2}
\end{align}
for $m = 0,1,2,\ldots$, beginning with $p^{0} = u^{0} = 0$.
In order to simplify the computation, we introduce
\begin{align}\label{3_redefF}
F_{m}(u) = \int_{\Omega} |\nabla u|^{\tilde{p}^{m}(x)} dx,
\end{align}
where $\tilde{p}^{m}(x) = P_{M}(|\nabla(G_{\tilde{\delta}} * u^{m})(x)|^{2})$.
Instead of $F(u)$ ($\tilde{p}$ dependent on $u$) by $F_{m}(u)$ defined above,
the Euler-Lagrange equation will be significantly simplified for each step as illustrated in (\ref{1Euler1}) and (\ref{1Euler2}).
Through this simplification, we can still capture the change of edges during evolution process
by just solving a simple Euler-Lagrange equation.

Using the definition of Bregman distance (\ref{3_BregmanDistance}), problem (\ref{3_ApproBregman}) will becomes
\begin{align}\label{3_ApproBregmanCon}
u^{m+1} = \argmin_{u} F_{m}(u) - F_{m}(u^{m}) - <p^{m}, u - u^{m}>
+ \frac{\lambda}{2} \|\hat{S}\hat{u} - \hat{g}^{\delta}\|_{L^{2}(\mathbb{R}^{2})}^{2}.
\end{align}
From the results in Section 2, it is easy to find that $u_{1}$ is well defined.
By similar arguments as deriving (3.4),(3.5) and (3.6) in \cite{WangLiu2013}, we can also
deduce that $\{ u^{m} : m \in \mathbb{N} \}$ is well defined.
Now let us firstly provide a recursive procedure which can solve (\ref{3_ApproBregman}) numerically in Algorithm 1.
\begin{algorithm}
\caption{Recursive Procedure}
\label{alg:A}
\begin{algorithmic}
\STATE {set
\begin{align*}
& (u^{0}, p^{0}) = (0,0)    \\
& F_{0}(u) = \int_{\Omega} |\nabla u|^{P_{M}(|\nabla(G_{\tilde{\delta}} * g^{\delta})(x)|^{2})} dx  \\
& u^{1} = \argmin_{u} F_{0}(u) + \frac{\lambda}{2}\|\hat{S}\hat{u} - \hat{g}^{\delta}\|_{L^{2}(\mathbb{R}^{2})}^{2}   \\
& \hat{f}^{1} = \hat{g}^{\delta} - \hat{S}\hat{u}^{1}
\end{align*}
}
\REPEAT
\STATE{
\begin{align*}
& u^{m+1} = \argmin_{u} F_{m}(u) + \frac{\lambda}{2}\|\hat{S}\hat{u} - \hat{g}^{\delta} - \hat{f}^{m}\|_{L^{2}(\mathbb{R}^{2})}^{2}   \\
& \hat{f}^{m+1} = \hat{f}^{m} + \hat{g}^{\delta} - \hat{S}\hat{u}^{m+1}
\end{align*}
}
\UNTIL{Some stopping condition is satisfied}
\end{algorithmic}
\end{algorithm}
Then, let us prove the properties of $\{u^{m}: m\in \mathbb{N}\}$ appeared in Algorithm 1
and the stoping criterion for iteration based on (\ref{3_ApproBregmanCon}).
Since the regularizing term changed during iteration in our case, the proof is more complex than
the TV regularization case.
\begin{theorem}\label{3_decrease}
For any fixed $\lambda > 0$, the data fitting error from the iteration is non-increasing, i.e.,
\begin{align}\label{3_th1de1}
\|\hat{S}\hat{u}^{m+1} - \hat{g}^{\delta}\|_{L^{2}(\mathbb{R}^{2})} \leq \|\hat{S}\hat{u}^{m} - \hat{g}^{\delta}\|_{L^{2}(\mathbb{R}^{2})},
\quad m = 1,2,\ldots
\end{align}
Moreover, it follows that
\begin{align}\label{2_th1de2}
\|\hat{S}\hat{u}^{M} - \hat{g}^{\delta}\|^{2}_{L^{2}(\mathbb{R}^{2})} \leq \frac{2}{\lambda M}(\|\nabla u_{*}\|_{L^{2}} + |\Omega|) + \delta^{2},
\end{align}
where $u_{*}$ is the minimizer of functional $\|\hat{S}\hat{u} - \hat{g}^{\delta}\|^{2}_{L^{2}(\mathbb{R}^{2})}$, $\delta > 0$
is the known error level.
\end{theorem}
\begin{proof}
Obviously, we have
\begin{align}\label{3_th1proof1}
\begin{split}
\frac{\lambda}{2}\|\hat{S}\hat{u}^{m+1} - \hat{g}^{\delta}\|_{L^{2}}^{2}
& \leq D_{F_{m}}^{p^{m}}(u^{m+1}, u^{m}) + \frac{\lambda}{2}\|\hat{S}\hat{u}^{m+1} - \hat{g}^{\delta}\|_{L^{2}}^{2}     \\
& \leq D_{F_{m}}^{p^{m}}(u^{m}, u^{m}) + \frac{\lambda}{2}\|\hat{S}\hat{u}^{m} - \hat{g}^{\delta}\|_{L^{2}}^{2} \\
& = \frac{\lambda}{2}\|\hat{S}\hat{u}^{m} - \hat{g}^{\delta}\|_{L^{2}}^{2}.
\end{split}
\end{align}
By direct computations, we can obtain
\begin{align}\label{3_th1proof2}
\begin{split}
& D_{F_{m}}^{p^{m}}(u,u^{m}) - D_{F_{m-1}}^{p^{m-1}}(u,u_{m-1}) + D_{F_{m-1}}^{p^{m-1}}(u^{m}, u^{m-1})   \\
= & F_{m}(u) - F_{m}(u^{m}) + <p^{m}, u^{m} - u> - F_{m-1}(u) + F_{m-1}(u^{m-1})    \\
& - <u_{m-1} - u, p_{m-1}> + F_{m-1}(u_{m}) - F_{m-1}(u_{m-1})  \\
& + <u_{m-1}-u_{m}, p^{m-1}>    \\
= & F_{m}(u) - F_{m}(u^{m}) - F_{m-1}(u) + F_{m-1}(u^{m}) + <u^{m} - u, p^{m} - p^{m-1}>.
\end{split}
\end{align}
Summing up the following estimates into (\ref{3_th1proof2})
\begin{align}\label{3_th1proof3}
\begin{split}
<u^{m} - u, p^{m} - p^{m-1}> & = \frac{\lambda}{2}<u^{m} - u, - \mathcal{F}^{-1}(\partial_{\hat{u}}
\|\hat{S}\hat{u} - \hat{g}^{\delta}\|_{L^{2}}^{2}|_{\hat{u} = \hat{u}^{m}})>    \\
& = \frac{\lambda}{2}<\hat{u}^{m}-\hat{u}, -\partial_{\hat{u}}
\|\hat{S}\hat{u} - \hat{g}^{\delta}\|_{L^{2}}^{2}|_{\hat{u} = \hat{u}^{m}}> \\
& \leq \frac{\lambda}{2}(\|\hat{S}\hat{u} - \hat{g}^{\delta}\|_{L^{2}}^{2} - \|\hat{S}\hat{u}^{m} - \hat{g}^{\delta}\|_{L^{2}}^{2}),
\end{split}
\end{align}
then we have
\begin{align}\label{3_th1proof4}
\begin{split}
D_{F_{m}}^{p^{m}}(u,u^{m}) & - D_{F_{m-1}}^{p^{m-1}}(u,u_{m-1}) + D_{F_{m-1}}^{p^{m-1}}(u^{m}, u^{m-1})  \\
& \leq F_{m}(u) - F_{m}(u^{m}) - F_{m-1}(u) + F_{m-1}(u^{m})    \\
& \quad +
\frac{\lambda}{2}(\|\hat{S}\hat{u} - \hat{g}^{\delta}\|_{L^{2}}^{2} - \|\hat{S}\hat{u}^{m} - \hat{g}^{\delta}\|_{L^{2}}^{2}).
\end{split}
\end{align}
Take $u = u_{*}$ in (\ref{3_th1proof4}) and rewrite it as
\begin{align}\label{3_th1proof5}
\begin{split}
D_{F_{m}}^{p^{m}}(u_{*}, u^{m}) + \frac{\lambda}{2}\|\hat{S}\hat{u}^{m} - \hat{g}^{\delta}\|_{L^{2}}^{2}
\leq & F_{m}(u) - F_{m}(u^{m}) - F_{m-1}(u)     \\
& + F_{m-1}(u^{m}) + D_{F_{m-1}}^{p^{m-1}}(u_{*}, u^{m-1})  \\
& - D_{F_{m-1}}^{p^{m-1}}(u^{m},u^{m-1}) + \frac{\lambda}{2}\delta^{2}.
\end{split}
\end{align}
Taking summation for $m = 1, \ldots, M$ yields
\begin{align}\label{3_th1proof6}
D_{F_{M}}^{p^{M}}(u_{*},u^{m}) + \frac{\lambda}{2}\sum_{m=1}^{M}\|\hat{S}\hat{u}^{m} - \hat{g}^{\delta}\|_{L^{2}}^{2}
\leq F_{M}(u_{*}) + \frac{\lambda}{2}M\delta^{2},
\end{align}
where we used $u^{0} = p^{0} = 0$.
Considering $D_{F_{M}}^{p^{M}}(u_{*},u^{m}) \geq 0$, we finally arrive at
\begin{align}\label{3_th1proof7}
\begin{split}
\|\hat{S}\hat{u}^{M} - \hat{g}^{\delta}\|_{L^{2}}^{2} \leq \frac{2}{\lambda M}(\|\nabla u_{*}\|_{L^{2}} + |\Omega|)
+ \delta^{2}.
\end{split}
\end{align}
\end{proof}

\begin{theorem}\label{3_stopth}
Assume that $u$ is the exact initial distribution. The optimal strategy for regularizing parameters $(\lambda, M)$
is that
\begin{align}\label{3_th21}
\frac{2}{\lambda M}(\|\nabla u_{*}\|_{L^{2}} + |\Omega|) = \delta^{2}.
\end{align}
For such a strategy, we have the optimization convergence rate
\begin{align}\label{3_th22}
\|\hat{S}(\hat{u}^{M} - \hat{u})\|_{L^{2}(\mathbb{R}^{2})} \leq (\sqrt{2} + 1)\delta \quad \text{as } \delta \rightarrow 0.
\end{align}
\end{theorem}
\begin{proof}
The results can be obtained by using Theorem \ref{3_decrease} and the following estimate
\begin{align*}
\|\hat{S}(\hat{u}^{M} - \hat{u})\|_{L^{2}}
\leq \|\hat{S}\hat{u}^{M} - \hat{g}^{\delta}\|_{L^{2}} + \|\hat{S}\hat{u} - \hat{g}^{\delta}\|_{L^{2}}
\leq \|\hat{S}\hat{u}^{M} - \hat{g}^{\delta}\|_{L^{2}} + \delta.
\end{align*}
\end{proof}

Since $\|\nabla u_{*}\|_{L^{2}}$ is unknown, the strategy (\ref{3_th21}) can not be implemented numerically
for choosing $(\lambda, M)$. An implementable scheme is that
we fix $\lambda$, then choosing the iteration stopping value $M$ such that
\begin{align}\label{3_practicalStop}
\|\hat{S}\hat{u}^{M} - \hat{g}^{\delta}\|_{L^{2}(\mathbb{R}^{2})} \leq \tau \delta
\end{align}
is satisfied first time for some specified $\tau > 1$.


For our problem, we must consider how to evaluate $\tilde{p}(x)$ during each iteration.
In order to simplify the computation and reduce the number of parameters needed to be specify,
here we use a surrogate for
$\tilde{p}(x) = P_{M}(|\nabla(G_{\tilde{\delta}} * u)(x)|^{2})$.
The essential idea of $\tilde{p}(x)$ is that its value is $1$ near the edges and its value is $2$
away from the edges. So we can using some simple algorithm to detect the edges firstly, then
based on the estimated edges build our exponent $\tilde{p}(x)$.
Denote $u^{m}$ to be the input of the $(m+1)$s iterate, we need to specify $\tilde{p}(x)$ from $u^{m}$.
Using some simple edge detection algorithm, we will obtain the estimated edges denoted as $E(u^{m})$ which
is an image with value $1$ on the detected edges and with value $0$ away from the edges.
Then, we define
\begin{align}\label{3_detectEdges}
\tilde{p}^{m}(x) = 2 I - G_{\tilde{\delta}} * E(u^{m})
\end{align}
where $I$ is the matrix which has the same dimension as $u^{m}$ and each element of $I$ is equal to $1$,
$G_{\tilde{\delta}}$ defined as in (\ref{1defp}) is a smoothing kernel.

Now we are ready to consider the algorithm for solving
\begin{align}\label{3_inarg0}
u^{m+1} = \argmin_{u} F_{m}(u) + \frac{\lambda}{2}\|\hat{S}\hat{u} - \hat{g}^{m}\|_{L^{2}}^{2}
\end{align}
with $\hat{g}^{m} = \hat{g}^{\delta} + \hat{f}^{m}$ for $m = 0,1,\ldots$.
Firstly, we introduce a new function $w$ to represent the gradient term $\nabla u$ in optimization problem (\ref{3_inarg0}),
which generates an equivalent constrained convex optimization problem:
\begin{align}\label{3_inarg1}
\min_{u,w} \left\{ \int_{\Omega} |w|^{\tilde{p}^{m}(x)} dx + \frac{\lambda}{2} \|\hat{S}\hat{u} - \hat{g}^{m}\|_{L^{2}}^{2} \right\}
\,\, \text{such that}\,\, w(x) = \nabla u(x).
\end{align}

Secondly, we split the domain $\Omega$ into three parts, for some small constant $\epsilon > 0$,
\begin{align}\label{3_inarg2}
\begin{split}
& \Omega_{1} := \{ x \in \Omega \, : \, 1 \leq \tilde{p}^{m}(x) < 1+ \epsilon \},   \\
& \Omega_{2} := \{ x \in \Omega \, : \, 1 + \epsilon \leq \tilde{p}^{m}(x) \leq 2 - \epsilon \},    \\
& \Omega_{3} := \{ x \in \Omega \, : \, 2 - \epsilon < \tilde{p}^{m}(x) \leq 2 \}.
\end{split}
\end{align}
Approximately, we can take $\tilde{p}(x) = 1$ on $\Omega_{1}$ and $\tilde{p}(x) = 2$ on $\Omega_{3}$.
Hence, we can rewrite (\ref{3_inarg1}) as follows
\begin{align}\label{3_inarg3}
\begin{split}
& \min_{u,w} \left\{ \int_{\Omega_{1}} |w| dx + \int_{\Omega_{2}} |w|^{\tilde{p}^{m}(x)} dx + \int_{\Omega_{3}} |w|^{2} dx
+ \frac{\lambda}{2} \|\hat{S}\hat{u} - \hat{g}^{m}\|_{L^{2}}^{2} \right\} \\
& \text{such that }w(x) = \nabla u(x).
\end{split}
\end{align}
Based on the domain decomposition, we define $w_{1} := w|_{\Omega_{1}}$, $w_{2} := w|_{\Omega_{2}}$ and $w_{3} := w|_{\Omega_{3}}$.

Thirdly, by using the splitting technique, we construct an iterative procedure of alternately solving a pair of easy subproblems.
The first three subproblems can be called '$w$-subproblem' for fixed $u = u^{*}$:
\begin{align}
& \argmin_{w_{1}}  \left\{ \int_{\Omega_{1}} |w_{1}| dx + \frac{\tilde{\lambda}}{2} \|w_{1} - \nabla u^{*}\|_{L^{2}(\Omega_{1})}^{2} \right\},
\label{3_subproblem1} \\
& \argmin_{w_{2}}  \left\{ \int_{\Omega_{2}} |w_{2}|^{\tilde{p}^{m}(x)} dx + \frac{\tilde{\lambda}}{2} \|w_{2} - \nabla u^{*}\|_{L^{2}(\Omega_{2})}^{2} \right\}, \label{3_subproblem2}     \\
& \argmin_{w_{3}}  \left\{ \int_{\Omega_{3}} |w_{3}|^{2} dx + \frac{\tilde{\lambda}}{2} \|w_{3} - \nabla u^{*}\|_{L^{2}(\Omega_{3})}^{2} \right\}.
\label{3_subproblem3}
\end{align}
The last subproblem is the '$u$-subproblem' for fixed $w = w^{*}$:
\begin{align}\label{3_subproblem4}
\argmin_{u} \left\{ \frac{\lambda}{2}\|\hat{S}\hat{u} - \hat{g}^{m}\|_{L^{2}(\Omega)}^{2}
+ \frac{\tilde{\lambda}}{2}\|w^{*} - \nabla u\|_{L^{2}(\Omega)}^{2} \right\}.
\end{align}

Using the definition of Frechet derivatives and standard computations,
we can easily obtain the minimizer of the subproblems (\ref{3_subproblem1}), (\ref{3_subproblem3}) and (\ref{3_subproblem4}).
Because the deduction is standard, we omit the details and just give the results as follows:
\begin{align}\label{3_subsolution1}
\begin{split}
w_{1}[u^{*}](x) = \left\{ \begin{array}{ll}
0 & x \notin \Omega_{1} \\
0 & x \in \Omega_{1} \text{ and } |\nabla u^{*}(x)| \leq \frac{1}{\tilde{\lambda}} \\
\left( |\nabla u^{*}(x)| - \frac{1}{\tilde{\lambda}} \right) \frac{\nabla u^{*}(x)}{|\nabla u^{*}(x)|}
& x \in \Omega_{1} \text{ and } |\nabla u^{*}(x)| > \frac{1}{\tilde{\lambda}}
\end{array} \right. ,
\end{split}
\end{align}
\begin{align}\label{3_subsolution3}
\begin{split}
w_{3}[u^{*}](x) = \left\{ \begin{array}{ll}
0 & x \notin \Omega_{3} \\
\frac{\nabla u^{*}(x)}{\tilde{\lambda} + 2} & x \in \Omega_{3}
\end{array} \right. ,
\end{split}
\end{align}
\begin{align}\label{3_subsolution4}
\hat{u}[w^{*}, g^{m}](\xi) = \frac{\lambda\hat{S}(\xi)\hat{g}^{m}(\xi) - i\tilde{\lambda}\xi\cdot\hat{w}^{*}}{\lambda \hat{S}^{2}(\xi)
+ \tilde{\lambda} |\xi|^{2}}.
\end{align}

For subproblem (\ref{3_subproblem2}), by a simple calculation, we can obtain the Euler-Lagrange equation
\begin{align}\label{3_sub2EulerLagrange}
\begin{split}
0 = -\nabla \cdot \left( \frac{w_{2}}{|w_{2}|} \tilde{p}(x) |w_{2}|^{\tilde{p}(x) - 1} \right) + \tilde{\lambda}(w_{2} - \nabla u^{*})
\end{split}
\end{align}
Denote $$J_{sub2}(w_{2}) := \int_{\Omega_{2}} |w_{2}|^{\tilde{p}^{m}(x)} dx +
\frac{\tilde{\lambda}}{2} \|w_{2} - \nabla u^{*}\|_{L^{2}(\Omega_{2})}^{2}.$$
Taking $\tilde{g}^{m}_{ij}$ to be a standard finite difference approximation of the right hand side of (\ref{3_sub2EulerLagrange})
at $x_{i,j}$ and $t_{m}$, we get an Euler-like updating scheme
\begin{align}\label{3_sub2ite}
w_{2\, i,j}^{m+1}[u^{*}] = w_{2\, i,j}^{m}[u^{*}] - \Delta t \tilde{g}_{i,j}^{m}.
\end{align}
Here we use an adaptive step size scheme. The new value $w_{2\, i,j}^{m+1}[u^{*}]$ is accepted for each step
in which the cost is improved, $J_{sub2}(w_{2\, i,j}^{m+1}[u^{*}]) < J_{sub2}(w_{2\, i,j}^{m}[u^{*}])$, and the step $\Delta t$
is increased by a factor $\Delta t \rightarrow (1+s)\Delta t$, $s > 0$. For each unsuccessful step where
$J_{sub2}(w_{2\, i,j}^{m+1}[u^{*}]) \geq J_{sub2}(w_{2\, i,j}^{m}[u^{*}])$, the trial step is not used, and the step size is
decreased, $\Delta t \rightarrow (1-s)\Delta t$.

In order to solve optimization problem (\ref{3_inarg1}) by solving subproblems from (\ref{3_subproblem1}) to (\ref{3_subproblem4}),
we need to solve subproblems from (\ref{3_subproblem1}) to (\ref{3_subproblem4}) iteratively with may times to obtain
an accurate solution. However, as mentioned in \cite{GoldsteinOsher2009}, we actually only need to solve
these subproblems with few iterations. Hence, we may not need to solve subproblem (\ref{3_subproblem2}) with very
high accuracy. That is to say we can run the iterative procedure (\ref{3_sub2ite}) with few steps.

At last, we state the discrete version of gradient operator and frequency operation.
For a function $f$, the discrete version of $\nabla f$ is $(\nabla f)_{i,j} := ((\nabla f)_{i,j}^{1}, (\nabla f)_{i,j}^{2})$ with
\begin{align*}
(\nabla f)_{i,j}^{1} = \left\{ \begin{array}{ll}
\frac{f_{i+1,j} - f_{i,j}}{\delta x_{1}} & 1 \leq i < N \\
0 & i = N
\end{array} \right. , \quad
(\nabla f)_{i,j}^{2} = \left\{ \begin{array}{ll}
\frac{f_{i,j+1} - f_{i,j}}{\delta x_{2}} & 1 \leq j < N \\
0 & j = N.
\end{array} \right.
\end{align*}
For simplicity, we assume that $\Omega := [-L,L]\times [-L,L]$ is a square, which yields that
$\delta x_{1} = \delta x_{2} := \delta x = \frac{2L}{N}$.
Then from Shannon-Nyquist sampling principle, the maximum frequency from the spatial grids is $[-\Omega_{0}, \Omega_{0}]$ with
\begin{align*}
\Omega_{0} = \frac{2\pi}{\delta x} = \frac{\pi N}{L}.
\end{align*}
We can compute the discrete Fourier transform in $[-\Omega_{0}, \Omega_{0}] \times [-\Omega_{0}, \Omega_{0}]$ with
uniform frequency distribution $\{\xi_{m,n}: m,n = 1,2,\ldots,N\}$.

Under these considerations, the iterative scheme for solving the optimization problem for the backward
time-space fractional diffusion model can be implemented by the Bregman iterative algorithm with some modifications.
For the details, see Algorithm 2.
\begin{algorithm}
\caption{Modified Bregman Iteration Algorithm}
\label{alg:B}
\begin{algorithmic}
\REQUIRE
$g^{\delta}, \delta, \lambda, \tilde{\lambda}, \tau, m_{\text{max}}, k_{\text{max}}, \ell_{\text{max}},
G_{\tilde{\delta}}, E(\cdot) \text{ is some edge detector},\Delta t, \epsilon, s, \text{tol}$
\STATE {set $u^{0} = p^{0} = 0, g^{0} = g^{\delta}, m = 0$}
\WHILE{$\|\hat{S}\hat{u}^{m} - \hat{g}^{\delta}\|_{L^{2}} \geq \tau \delta$ and $m \leq m_{\text{max}}$}
\STATE $u^{m,0} = u^{m} \quad \tilde{p}^{m} = 2I - G_{\tilde{\delta}} * E(u^{m})$
\FOR{$k=0$ to $k_{\text{max}}$}
\STATE $w_{1}^{k+1}\longleftarrow w_{1}[u^{m,k}]$
\WHILE{$\ell \leq \ell_{\text{max}}$ and $\|w_{2}^{\ell + 1} - w_{2}^{\ell}\|_{L^{2}} > \text{tol}$}
\STATE $\tilde{g}^{\ell} = -\nabla \cdot \left( \frac{w^{\ell}_{2}}{|w^{\ell}_{2}|} \tilde{p}^{m}
|w^{\ell}_{2}|^{\tilde{p}^{m} - 1} \right) + \tilde{\lambda}(w^{\ell}_{2} - \nabla u^{m,k})$
\IF{$J_{sub2}(w_{2}^{\ell}) > J_{sub2}(w_{2}^{\ell} - \Delta t \tilde{g}^{\ell})$}
\STATE $w_{2}^{\ell + 1} \longleftarrow w_{2}^{\ell} - \Delta t \tilde{g}^{\ell}$, $\Delta t \longleftarrow (1+s)\Delta t$
\ELSE
\STATE $\Delta t \longleftarrow (1-s)\Delta t$
\ENDIF
\ENDWHILE
\STATE $w_{3}^{k+1} \longleftarrow w_{2}[u^{m,k}]$, $w^{k+1} \longleftarrow w^{k+1}_{1} + w^{k+1}_{2} + w^{k+1}_{3}$
\STATE $u^{m,k+1} \longleftarrow u[w^{k+1}, \hat{g}^{m}]$
\ENDFOR
\STATE $u^{m+1} \longleftarrow u^{m,k+1}$, $\hat{g}^{m+1} \longleftarrow \hat{g}^{m} + (\hat{g}^{\delta} - \hat{S}\hat{u}^{m+1})$
\ENDWHILE
\ENSURE
$u^{m+1}$
\end{algorithmic}
\end{algorithm}

\section{Numerical Examples}

In this section, we consider two typical examples. In these two examples, we will compare
our results with TV regularizing and Tikhonov regularizing model. Here we first list the two models as follows
\begin{align}
& u_{TV} = \argmin_{u} \left\{ \|u\|_{TV} + \frac{\lambda}{2} \|\hat{S}\hat{u} - \hat{g}^{\delta}\|_{L^{2}(\mathbb{R}^{2})}^{2} \right\},
\label{4_TVmodel}  \\
& u_{Tik} = \argmin_{u} \left\{ \|\nabla u\|^{2}_{L^{2}} + \frac{\lambda}{2} \|\hat{S}\hat{u} - \hat{g}^{\delta}\|_{L^{2}(\mathbb{R}^{2})}^{2} \right\},
\label{4_TikModel}
\end{align}
where $\lambda$ is the regularization parameter, $g^{\delta}$ is the measured data with noise.
For the the TV regularizing model, we refer to \cite{OsherBurger2005} which described clearly how to solve
TV regularization model. For Tikhonov regularization model, it can be solved just by a small modification
of algorithm stated in \cite{OsherBurger2005}. More explicitly, we just need to change the Euler-Lagrange equation of problem
(\ref{4_TVmodel}) by the Euler-Lagrange equation of problem (\ref{4_TikModel}).

Here, we specify some parameters used in our implementation.
In our examples, we take parameters in Algorithm 2 as follows
\begin{align*}
& \tilde{\delta} = 0.4, \quad \tau = 1.01, \quad m_{\text{max}} = 500, \quad
k_{\text{max}} = 2, \quad \ell_{\text{max}} = 5, \quad \text{tol} = 10^{-6},    \\
& s = 0.1, \quad \Delta t = 0.1, \quad \epsilon = 0.1,  \\
& E(\cdot) \text{ to be the Canny edge detection algorithm in the Matlab toolbox}.
\end{align*}
For the noise, we take $\delta = 0.0005$ and $\delta = 0.005$ respectively.
Because the value of $\lambda$ can determine the convergence rate of our algorithm, we take different $\lambda$
for different noise level. If we take $\lambda$ too big, $\|Su_{1} - \hat{g}^{\delta}\|_{L^{2}}$ may less
than $\tau \delta$ when the first iteration finished. In this case, we may incorporate more noise in our result $u_{1}$.
If we take $\lambda$ too small, the iteration will converge too slow to obtain our final results, e.g. exceed $500$ steps.
For $\delta = 0.0005$ and $\delta = 0.005$,
we take $\lambda = 10^{11}$ and $\lambda = 10^{9}$ respectively in our numerical experiments.

In order to avoid the error in solving the forward fractional differential equation, we solve (\ref{1_timespacefrac}) to obtain
the solution at time $T$ using the Laplace transform
\begin{align*}
\hat{v}(T,\xi) = E_{\alpha,1}(-|\xi|^{\beta}T^{\alpha})\hat{u}(\xi),
\end{align*}
where the Mittag-Leffler function $E_{\alpha,1}(\cdot)$ is numerically calculated up to desired accuracy by
standard algorithm provided by Podlubny \cite{PodlubnyMatlab}.

Denote $x = (x_{1},x_{2}) \in \mathbb{R}^{2}$ and $\xi = (\xi_{1},\xi_{2}) \in \mathbb{R}^{2}$.
We generate the final measurement data with noise by
\begin{align*}
g^{\delta} := v^{\delta}(T,x) = \mathcal{F}^{-1}(\hat{v}(T,\xi)) + \delta \cdot \text{randn}(x) \cdot \max(\mathcal{F}^{-1}(\hat{v}(T,\xi))),
\end{align*}
where randn is the pseudo-random number generating from the standard normal distribution.
Notice that here we add noise as in \cite{BuiThanhGhattas2015} where the $\delta$ stands for the noise level is $100 \times \delta$,
e.g. when $\delta = 0.05$ the noise level is $5\%$.

In our discretization, we discrete $\Omega = [-10,10]^{2}$, the support of $u(x)$, by uniform grids
$(x_{1}(i),x_{2}(j)) \in [-10,10]^2$ with $i,j = 1,\ldots,256$.

We use relative error (RelErr) to quantitatively compare our solution with those based on
TV regularization and Tikhonov regularization. For given finite dimensional vectors $g$ and its
noisy form $g^{\delta}$ representing the image, the above RelErr has the representation
\begin{align}\label{xiangdui}
\text{RelErr}(g^{\delta},g) := \frac{\|g^{\delta} - g\|_{L^{2}(\Omega)}}{\|g\|_{L^{2}(\Omega)}} \times 100\%.
\end{align}

\textbf{Example 1.} We consider $u(x) = e^{-|x|^{2}}$, $T = 1$.
In this case, the exact solution has the following form
\begin{align}\label{4_solexp}
\hat{v}(T,\xi) = \pi E_{\alpha,1}(-|\xi|^{\beta}T^{\alpha})\hat{u}(\xi) e^{-|\xi|^{2}}.
\end{align}
We take $\alpha = 0.6$, $\beta = 1$ and $T = 1$ to see the difference between the three different models.
In table 1, relative error defined in (\ref{xiangdui}) for three different methods are presented.
Because the noise added by random algorithms, we run the three different algorithms 100 times and the data are the
averages.
\begin{center}
\begin{table}[htbp]\label{biao1}
\caption{The values of RelErr of three methods for Example 1}
\begin{tabular}{c|c|c|c}
  \hline
  RelErr  & TV model & Tikhonov model & Variable TV model  \\
  \hline
  $\sigma = 0.0005$ & $3.8283\%$ & $0.3857\%$ & $0.3696\%$ \\
  $\sigma = 0.005$ & $8.8646\%$ & $0.6559\%$ & $0.6597\%$ \\
  \hline
\end{tabular}
\end{table}
\end{center}
From table 1, we could clearly know that TV model's performance is much weaker than the Tikhonov model
and our variable TV model's performance is comparable to the Tikhonov model.
Because the differences for recovered functions obtained by different methods can not see clearly form the figures of the recovered function,
we will not provide the comparison figures for the recovered function and only provide the original data,
recovered data with $\delta = 0.0005$ and $\delta = 0.005$ in figure \ref{exp1tu} which show that the recovered data
have no visual difference with the original data.
\begin{figure}[htbp]
\centering
\includegraphics[width=1\textwidth]{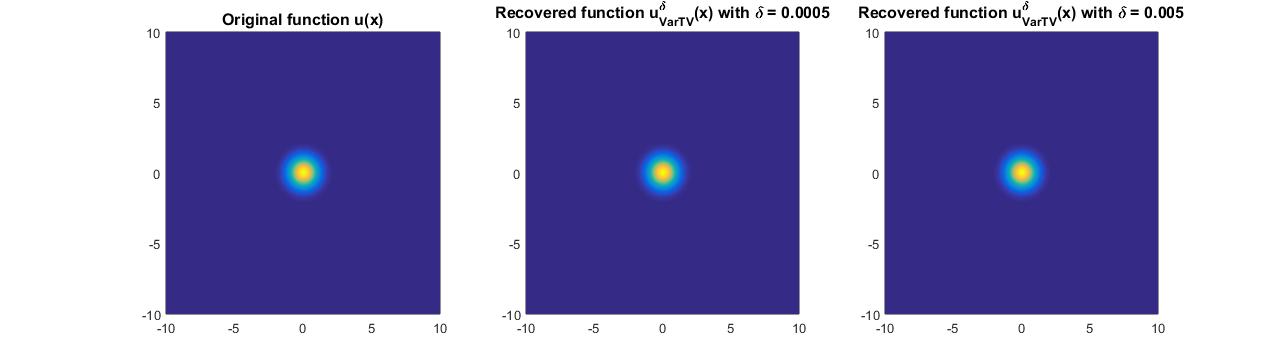}
\caption{Left:~Original function; Middle:~Recovered function by variable TV model with $\delta = 0.0005$;
Right:~Recovered function by variable TV model with $\delta = 0.005$ for Example 1.}\label{exp1tu}
\end{figure}

\textbf{Example 2.} Consider a phantom model generated by standard function phantom.m in Matlab with defalut
parameters. We use the gray level (piecewise constant) of this image as the values of $u(x)$, see figure \ref{4_2originalhou}.
\begin{figure}[htbp]
\centering
\includegraphics[width=0.25\textwidth]{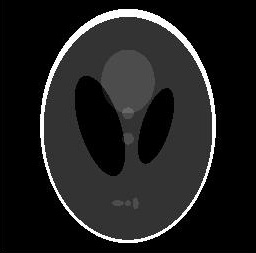}
\caption{Initial function for Example 2.}\label{4_2originalhou}
\end{figure}
In this example, we take $\alpha = 0.6$ and $\beta = 0.9$ and $N = 256$.
In the following, we provide table 2 to present the performance of the three different models.
As in example 1, we also run the three different algorithms 100 times and the data in table 2 are the averages.
Table 2 demonstrate that the TV model's performance is better than Tikhonov model when the initial data is
a piecewise constant function. Our variable TV model as expected preform comparable to the TV model.
Hence, example 1 and example 2 reflect that our model can change the value of $\tilde{p}$ and the algorithm proposed in
section 3 can solve our variable TV regularization model effectively.
Based on same considerations as stated in example 1, we will not present the three different figures of the recovered function and
only present the recovered functions of our variable TV model in figure \ref{exp2tu} which show that
the recovered functions are much similar to the original data.
\begin{center}
\begin{table}[htbp]\label{biao2}
\caption{The values of RelErr of three methods for Example 2}
\begin{tabular}{c|c|c|c}
  \hline
  RelErr  & TV model & Tikhonov model & Variable TV model  \\
  \hline
  $\sigma = 0.0005$ & $13.0053\%$ & $13.7772\%$ & $13.0666\%$ \\
  $\sigma = 0.005$ & $22.7222\%$ & $25.2101\%$ & $22.7810\%$ \\
  \hline
\end{tabular}
\end{table}
\end{center}
\begin{figure}[htbp]
\centering
\includegraphics[width=0.9\textwidth]{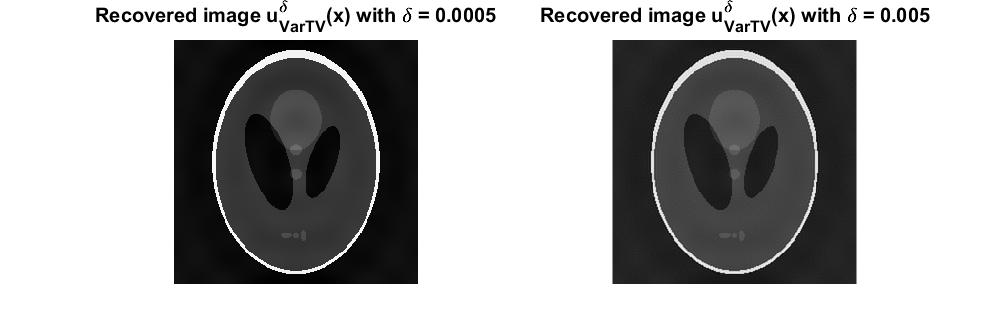}
\caption{Left:~Recovered function by variable TV model with $\delta = 0.0005$ for Example 2;
Right:~Recovered function by the variable TV model with $\delta = 0.005$ for Example 2.}\label{exp2tu}
\end{figure}

Now, we provide a simple verification of our theoretical results.
Here, we take $\delta = 0.0005$.
From theorem \ref{3_stopth}, we know that
\begin{align*}
\lambda M = \frac{2}{\delta^{2}}(\|\nabla u_{*}\|_{L^{2}} + |\Omega|) = C
\end{align*}
for some unknown constant C.
If we take $\lambda = 10^{11}, \frac{1}{4}\times 10^{11}, \frac{1}{16} \times 10^{11}$ respectively.
We run our program by taking $\tau = 1.01$ and using our stop criterion (\ref{3_practicalStop})
to obtain the iterative step $M = 10$ when $\lambda = 10^{11}$.
If $M = 10$ is accurate, $M$ should be equal to $40$ and $160$ when $\lambda = \frac{1}{4}\times 10^{11}$
and $\frac{1}{16} \times 10^{11}$ respectively according to our theory.
We run our program and obtain $M$ equal to $38$ and $167$
when $\lambda = \frac{1}{16}\times 10^{11}$ and $\lambda = \frac{1}{4}\times 10^{11}$ respectively.
We can see that it is almost the same as the predicted by the theoretical results which shows that
our program is right and in accordance with our theories.
Here is the result after run our algorithm once, each time the result will be a little different for
the noise is added randomly.

After theoretical justifications, we want to clarify an interesting phenomena which reveals
some essential different properties of the inverse problems for integer-order differential equations
and fractional-order differential equations.

\textbf{Discontinuous for normal and anomalous diffusion.}
In this part, we also use the phantom model generated by standard function phantom.m in Matlab with defalut
parameters (same as in example 2) as our initial data then
take $T = 1$, $\sigma = 0.0005$, $\beta = 1$ and the time derivative $\alpha = 0.5, \ldots, 1$.
In order to provide a clear explanation, we take $100$ points between $[0.5,1]$ for $\alpha$.
Then we use our variable TV regularizing model to recover the true initial data and plot the RelErr value
for each $\alpha$ in figure \ref{4_2TimeFracDis}.
\begin{figure}[htbp]
\centering
\includegraphics[width=0.6\textwidth]{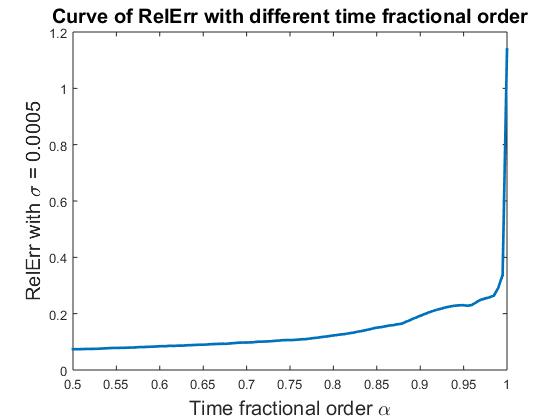}
\caption{The curve of the relative error of the recovered data for different values of parameter $\alpha$.}\label{4_2TimeFracDis}
\end{figure}
Because we used the same model, the degree of ill-posedness intuitively can be represented by the RelErr value.
Small RelErr value indicate that our model can provide a good result, hence, the degree of ill-posedness is weak.
In contradict, large RelErr value indicate that the degree of ill-posedness is strong.
From figure \ref{4_2TimeFracDis}, we clearly find that even for $\alpha = 0.99$ the degree of ill-posedness
is much weaker than the integer-order equation. This implies that for $\alpha < 1$ the degree of ill-posedness
varies continuously, however, for $\alpha = 1$ the degree of ill-posedness is much higher than any value of $\alpha < 1$.
The degree of ill-posedness may not change continuously at the point $1$.
This observation may be explained by the properties of Mittag-Leffler function $E_{\alpha,1}(z)$. For $\alpha = 1$,
it is an exponential function, however, for any value $\alpha < 1$ the Mittag-Leffler function behaves like polynomial
functions for large $z$ (Theorem 1.3 in \cite{Podlubny1999}).
This property also be observed in \cite{WangGaoZhou2013} which propose
a fractional extension of instantaneous frequency attribute to detect thin layers of sandstone formations.
They use fractional order of $0.99$ and illustrate only $0.01$ smaller than integer-order $1$ will
bring very different results.

From the above two typical examples, it is obviously that our algorithm behaves like $L^{2}$ based Tikhonov regularizing
model when the initial function is smooth and behaves like $TV$ regularizing model when the initial function is piecewise constant.
Hence, our model has more flexibility compared with Tikhonov regularizing model and $TV$ regularizing model.

\section{Acknowledgements}
J.~Gao was supported partially by the National Natural Science Foundation of China
under grant no. 41390454.
J.~Jia was supported by the National Natural Science Foundation of China under grant no. 11501439 and
the postdoctoral science foundation project of China under grant no. 2015M580826.
J.~Peng was supported partially by National Natural Science Foundation of China under grant no. 11131006 and by the National Basic Research Program of China under grant no. 2013CB329404.

\end{document}